\documentclass[11pt,reqno]{amsart}

\usepackage{amsthm,amsmath,amssymb}
\usepackage{ascmac}
\usepackage{fullpage}
\usepackage{graphicx}

\newcommand{\E}[0]{\mathbb{E}}

\newcommand{\R}[0]{\mathbb{R}}

\newcommand{\Prob}[0]{\mathbb{P}}

\renewcommand{\tilde}{\widetilde}
\renewcommand{\hat}{\widehat}
\renewcommand{\le}{\leqslant}
\renewcommand{\ge}{\geqslant}

\DeclareMathOperator{\Cov}{Cov}

\DeclareMathOperator{\Card}{Card} 

\DeclareMathOperator*{\argmin}{arg\,min}
\DeclareMathOperator*{\argmax}{arg\,max}

\theoremstyle{plain}
\newtheorem{theorem}{Theorem}
\newtheorem{lemma}{Lemma}
\newtheorem{proposition}{Proposition}
\newtheorem{corollary}{Corollary}
\newtheorem{assumption}{Assumption}

\theoremstyle{definition}

\newtheorem{remark}{Remark}

\begin{document}
\title[QR approach to modal regression]{Quantile regression approach to conditional mode estimation}

\author[H. Ota]{Hirofumi Ota}
\author[K. Kato]{Kengo Kato}
\author[S. Hara]{Satoshi Hara}

\date{The first arXiv version: November 13, 2018. This version: \today}
\thanks{This paper supersedes ``On estimation of conditional modes using multiple quantile regressions'' (Hirofumi Ohta and Satoshi Hara, arXiv:1712.08754).} 

\address[H. Ota]{
Graduate School of Economics, University of Tokyo \\
7-3-1 Hongo, Bunkyo-ku, Tokyo 113-0033, Japan
}
\email{hirofumi-ohta@g.ecc.u-tokyo.ac.jp}

\address[K. Kato]{
Department of Statistics and Data Science, Cornell University \\
1194 Comstock Hall, Ithaca, NY 14853, USA
}
\email{kk976@cornell.edu}

\address[S. Hara]{
The Institute of Scientific and Industrial Research, Osaka University \\
Mihogaoka 8-1, Ibaraki, Osaka 567-0047, Japan
}
\email{satohara@ar.sanken.osaka-u.ac.jp}

\begin{abstract}
In this paper, we consider estimation of the conditional mode of an outcome variable given regressors. To this end, we propose and analyze a computationally scalable estimator derived from a linear quantile regression model and develop asymptotic distributional theory for the estimator. Specifically, we find that the pointwise limiting distribution is a scale transformation of Chernoff's distribution despite the presence of regressors. 
In addition, we consider analytical and subsampling-based confidence intervals for the proposed estimator.
We  also conduct Monte Carlo simulations to assess the finite sample performance of the proposed estimator together with the analytical and subsampling confidence intervals.
Finally, we apply the proposed estimator to predicting the net hourly electrical energy output using Combined Cycle Power Plant Data.
\end{abstract}



\maketitle

\section{Introduction}

Estimation of the conditional mode of an outcome variable given regressors, called \textit{modal regression}, is an active research area in the recent statistics literature. In particular, if the conditional distribution is highly skewed or has fat tails, 
then one would be more interested in the conditional mode  than the conditional mean or median since in such cases the mean or median may fail to capture a major trend of  the conditional distribution. 
As such, modal regression has a wide variety of applications including the analysis of traffic and forest fire data \cite{EinbecTutz2006,YaoLi2014}, econometrics \cite{Lee1989,Lee1993,KempSantos2012,Ho2017}, and machine learning \cite{Sasaki2016,Feng2017}. For example, \cite{KempSantos2012} argue that the mode is the most intuitive measure of central tendency for positively skewed data found in many econometric applications such as  wages, prices, and expenditures (\cite{KempSantos2012}, p. 93). See also \cite{Chen2017} and \cite{Chacon2018} for recent reviews on modal regression. 

Existing approaches to estimation of the conditional mode includes nonparametric kernel estimation \cite{Chen2016} and linear modal regression \cite{Lee1989,Lee1993,KempSantos2012,YaoLi2014}, among others. 
The nonparametric estimation is able to avoid model misspecification but has slow rates of convergence that deteriorate as the number of regressors increases. Namely, if the number of continuous regressors is $p$, then the rate of convergence of the kernel density based estimator in \cite{Chen2016} is at best $n^{-2/(p+7)}$ under four times differentiability of the joint density.  On the other hand, the linear modal regression is able to avoid such ``curse of dimensionality'' but requires to solve a multi-dimensional non-convex optimization problem.  

In this paper, we propose a new estimator for the conditional mode that is able to avoid  the curse of dimensionality and at the same time is computationally scalable, thereby complementing the above existing methods.
The proposed method is based on the observation that the derivative of the conditional quantile function with respect to the quantile index is the reciprocal of the conditional density evaluated at the conditional quantile function and hence the conditional mode is obtained by minimizing the derivative of the conditional quantile function. 
Specifically, we assume a linear quantile regression model to estimate the conditional quantile function as in \cite{KoenkerBassett1978} (see also \cite{Koenker2005}), and estimate its derivative by a numerical differentiation of the estimated conditional quantile function. 
The proposed estimator is then obtained by minimizing the estimated derivative. 
Notably, the proposed method  is computationally attractive since the computation of the quantile regression estimate can be formulated as a linear programming problem and so is highly scalable (cf. Chapter 6 in \cite{Koenker2005}), and the minimization of the estimated derivative is a one-dimensional optimization problem and so can be carried out by a grid search. 

We develop asymptotic theory for the proposed estimator, which turns out to be non-standard. 
Specifically, we find that the proposed estimator has convergence rate $(nh^{2})^{-1/3}$ where $n$ is the sample size and $h=h_n \to 0$  is a sequence of bandwidths, and the limiting distribution is a scale transformation of Chernoff's distribution \cite{Chernoff1964}. Chernoff's distribution is defined as the distribution of a maximizer of a two-sided Brownian motion with a negative quadratic drift, and appears as e.g. limiting distributions of estimators for monotone functions; see \cite{GroeneboomWellner2001}.
Our result on the limiting distribution would be of interest from theoretical and practical perspectives. First, the proposed estimator provides a new example of estimators having Chernoff's distribution as limiting distributions, which would be of theoretical interest. 
Second, the fact that the limiting distribution is a scale transformation of Chernoff's distribution makes inference for our estimator relatively simple. 
This is in contrast to e.g. Manski's maximum score \cite{Manski1975} whose limiting distribution is a maximizer of a Gaussian process with its covariance function depending on the distribution of regressors; see \cite{KimPollard1990}. 
Building upon the limiting distribution, we develop inference methods for our estimator. 
The one is an analytical confidence interval based on consistently estimating the scaling constant, and the other is based on the subsampling \cite{PolitisRomano1994,Politis1999}. 
We also derive a multivariate limit theorem for the proposed estimator, which can be used to construct  simultaneous confidence intervals for the modal function over finite design points.

In addition to the theoretical results, we conduct Monte Carlo simulations to assess the finite sample performance of the proposed estimator together with the analytical and subsampling confidence intervals. 
We  suggest a practical method to choose the bandwidth based upon the idea suggested in \cite{KoenkerMachado1999}. We compare the performance of the proposed estimator with the linear modal regression estimator of \cite{KempSantos2012,YaoLi2014} via the root mean square error for the two data generating processes  where the true modal function is linear or nonlinear. 
Finally, we apply the proposed estimator to predicting the net hourly electrical energy output using Combined Cycle Power Plant Data \cite{Kaya2012,Tufekci2014}.
These numerical results show evidence that the proposed estimator works well in the finite sample.

The literature related to this paper is broad. Nonparametric estimation of the unconditional mode goes back to Parzen \cite{Parzen1962} and Chernoff \cite{Chernoff1964} in 1960s; see also \cite{Romano1988}.
Modal regression originates from \cite{SagerThisted1982} and the literature has flourished since then \cite{Lee1989,Lee1993,EinbecTutz2006,KempSantos2012,Yao2012,YaoLi2014,Chen2016,ZhouHuang2016,Sasaki2016,Ho2017,Krief2017,KhardaniYao2017,Feng2017}.
However, none of these papers do not consider a quantile regression based estimator for the conditional mode. 
\cite{Lee1989,Lee1993,KempSantos2012,YaoLi2014} consider linear modal regression; \cite{Lee1989,Lee1993} assume a restrictive condition that the conditional distribution is symmetric around the origin to derive limiting distributions of the estimators. The symmetry of the conditional distribution implies that  the conditional mean, median, and mode are all identical. Subsequently, \cite{KempSantos2012,YaoLi2014} relax the symmetry assumption and propose estimators that enjoy asymptotic normality. 
In the present paper, instead of linearity of the conditional mode, we assume a linear quantile regression model. Importantly, the linear quantile regression model does not imply linearity of the conditional mode, and so there are no strict inclusion relations between the two assumptions; see Remark \ref{rem: linearity} ahead. 
The recent work of \cite{Chen2016} studies nonparametric kernel estimation of the conditional mode. To be precise, \cite{Chen2016} do not assume the existence of the unique global mode and  allow for multiple local modes. Extension of our approach to multiple local modes would be of interest but is beyond the scope of the present paper. \cite{Yao2012} propose a local modal regression (LMR) estimator that can be seen as a local linear estimator for the conditional mode, and establish asymptotic results analogous to those of a local linear estimator for the conditional mean. In particular, the rate of convergence of the LMR estimator is faster than that of the kernel density based estimator of \cite{Chen2016}. This is, however, due to Condition (A6) in \cite{Yao2012} that is essentially the conditional symmetry assumption on the error term (note that $h_2$ in \cite{Yao2012} is fixed) and under which the conditional mode and mean coincide. In the present paper, we assume no symmetry assumptions on the conditional distribution.

From a technical point of view, derivation of the limiting distribution of the proposed estimator is by no means trivial.
First of all, it is not a priori straightforward to foresee that the convergence rate is $(nh^{2})^{-1/3}$ and the limiting distribution is a scale transformation of Chernoff's distribution. Second, because our objective function depends on the bandwidth tending to zero as the sample size increases, our result does not follow from the general theorem, Theorem 1.1, in \cite{KimPollard1990}, which is a pioneering work on cube root asymptotic theory. The recent work of \cite{SeoOtsu2018} extends \cite{KimPollard1990} to allow the objective function to depend on the bandwidth, but some of their regularity conditions are severely restrictive or difficult to verify in our problem. Hence, we provide a separate and self-contained proof of the main theorem, Theorem \ref{thm: limiting distributions} ahead, which requires  a substantial work. See also the discussion after Theorem \ref{thm: limiting distributions}.

The rest of the paper is organized as follows. In Section \ref{sec: setup}, we state the formal setup and define the estimator. In Section \ref{sec: limiting distributions}, we derive limiting distributions of the proposed estimator and develop inference methods for it. In Section \ref{sec: numerical results}, we conduct Monte Carlo simulations to assess the finite sample performance of the proposed estimator together with the analytical and subsampling confidence intervals. In addition,  we apply the proposed estimator to predicting the net hourly electrical energy output using Combined Cycle Power Plant Data.
Section \ref{sec: conclusion} concludes. 
All the proofs are gathered in Appendix. 

\section{Setup and estimator}
\label{sec: setup}

In this paper, we are interested in estimating the conditional mode of an outcome variable $Y \in \R$ given a vector of regressors $X = (X_{1},\dots,X_{d})^{T} \in \R^d$. 
In what follows, we assume that there exists a conditional density $f(y \mid x)$ of $Y$ given $X$ that is (at least) continuous in $y$, and for each design point $x$ in the support of $X$, there exists a unique mode $m (x)$, i.e., there exists a unique maximizer of the  function  $y \mapsto f(y \mid x)$:
\[
 f(m(x) \mid x) = \max_{y \in \R} f(y \mid x). 
\]
The function $m(x)$ is called the modal function. 

We base our estimation strategy of the modal function $m (x)$ on inverting a quantile regression model.  Let $Q (\tau \mid X)$ denote the conditional $\tau$-quantile of $Y$ given $X$ for $\tau \in (0,1)$. For the notational convenience, we also write $Q_{x}(\tau) = Q(\tau \mid X=x)$. To see the link between the conditional quantile function and the modal function, we begin with observing that 
\[
s_{x}(\tau) := Q_{x}'(\tau) = \frac{\partial Q_{x}(\tau)}{\partial \tau} = \frac{1}{f(Q_{x}(\tau) \mid x)}
\]
assuming some regularity conditions that will be clarified below. Hence, defining 
\[
\tau_{x} = \argmin_{\tau \in (0,1)} s_{x}(\tau),
\]
which exists and is unique (by continuity and strict positivity of the function $y \mapsto f(y \mid x)$ around the mode $m(x)$),  
we arrive at the key identity
\[
m(x) = Q_{x} (\tau_{x}).
\]
The function $\tau \mapsto s_{x}(\tau)$ (called the ``sparsity'' function) can be estimated by a numerical differentiation of an estimator of the conditional quantile function $\tau \mapsto Q_{x}(\tau)$, and so the problem boils down to estimating the conditional quantile function. To this end, we assume a linear quantile regression model:
\begin{equation}
Q (\tau \mid X) = X^{T} \beta(\tau), \ \tau \in (0,1),
\label{eq: QR model}
\end{equation}
where $\beta (\tau) \in \R^{d}$ is an unknown slope vector for each $\tau \in (0,1)$. 

Pick any design point $x$ in the support of $X$, and consider to estimate $m(x)$. 
Let $(Y_1,X_1),\dots,(Y_n,X_n)$ be i.i.d. observations of $(Y,X)$. We estimate the slope vector $\beta (\tau)$ by
\begin{equation}
\hat{\beta}(\tau) = \argmin_{\beta \in \R^d}  \sum_{i=1}^{n} \rho_{\tau}(Y_i-X_i^{T}\beta),
\label{eq: QR problem}
\end{equation}
where $\rho_{\tau} (u) = \{ \tau - I(u \le 0) \} u$ is the check function \cite{KoenkerBassett1978}. This leads to an estimator $\hat{Q}_{x}(\tau) = x^{T}\hat{\beta}(\tau)$ of $Q_{x}(\tau)$. To estimate $s_{x}(\tau) =Q_{x}'(\tau)$, let $h = h_{n} \to 0$ be a sequence of bandwidths such that $nh^{2} \to \infty$; then we estimate $s_{x}(\tau)$ by a numerical differentiation:
\[
\hat{s}_{x}(\tau) = \frac{\hat{Q}_{x}(\tau + h) - \hat{Q}_{x}(\tau-h)}{2h}.
\]
Finally, we estimate $m(x)$ by $\hat{m}(x) = \hat{Q}_{x}(\hat{\tau}_{x}) =x^{T}\hat{\beta}(\hat{\tau}_{x})$, where $\hat{\tau}_{x}$ is an approximate minimizer of $\hat{s}_{x}(\tau)$ on $[\varepsilon,1-\varepsilon]$ with sufficiently small parameter $\varepsilon \in (0,1/2)$ chosen by users, in the sense that  
\[
\hat{s}_{x}(\hat{\tau}_{x}) \le \inf_{\tau \in [\varepsilon,1-\varepsilon]} \hat{s}_{x}(\tau) + o((nh^{2})^{-2/3}). 
\]
The objective function $\hat{s}_{x}(\tau)$ may not admit strict minimizers, and so we allow $\hat{\tau}_{x}$ to be an approximate minimizer in the above sense, which always exists. In practice, our estimator requires to choose the bandwidth $h$, which will be discussed in Section \ref{subsec: bandwidth selection}. 

Importantly,  our estimate $\hat{m}(x)$ is easy to compute even when the sample size $n$ and the dimension $d$ of $X$ are large. The quantile regression problem (\ref{eq: QR problem}) can be formulated as a linear programming problem and hence can be efficiently solved even when $n$ and $d$ are large (cf. Chapter 6 in \cite{Koenker2005}). 
Furthermore, the entire path $\tau \mapsto \hat{\beta}(\tau)$ can be computed by a parametric linear programming or discretizing the interval $(0,1)$ into fine grids. 
The minimization of $\hat{s}_{x}(\tau)$ is a one-dimensional optimization problem and can be solved by a grid search. On the other hand, the linear modal regression estimator  \cite{Lee1989, Lee1993, KempSantos2012,YaoLi2014} requires to solve a multi-dimensional  non-convex optimization problem.
For example, \cite{YaoLi2014} assume that the modal function is linear $m(x) = x^{T}\gamma$ for some $\gamma \in \R^{d}$ and propose the following estimator:
\begin{equation}
\hat{\gamma}_{YL} = \argmax_{\gamma} \sum_{i=1}^{n} \phi_{h}(Y_{i}-X_{i}^{T}\gamma),
\label{eq: YaoLi}
\end{equation}
where $\phi (y) = (2\pi)^{-1/2} e^{-y^{2}/2}$ is the density of the standard normal distribution and $\phi_{h}(y) = h^{-1} \phi(y/h)$. The optimization problem (\ref{eq: YaoLi}) is non-convex. \cite{YaoLi2014} propose an EM like algorithm for (\ref{eq: YaoLi}), but ``there is no guarantee that the algorithm will converge to the global optimal
solution'' (\cite{YaoLi2014}, p. 659).

\begin{remark}[Generality of linear quantile regression model]
\label{rem: linearity}
The linear quantile regression model (\ref{eq: QR model}) is flexible enough to cover many data generating processes. In general, if $\tau \mapsto \beta (\tau)$ is a function on $(0,1)$ such that the map $\tau \mapsto X^{T}\beta (\tau)$ is strictly increasing almost surely and $Y$ is generated as $Y = X^{T} \beta(U)$ for $U \sim U(0,1)$ independent of $X$, then the pair $(Y,X)$ satisfies the linear quantile regression model (\ref{eq: QR model}). In particular, it is worth pointing out that 
the linear quantile regression model (\ref{eq: QR model}) does not imply linearity of the modal function $m(x)$. 
For example, consider the simple case where $X=(1,X_2)^{T}$ with $X_2 \in (0,1)$ and $Y = U^{3}/3- X_2 (U-1)^{2}$ for $U \sim U(0,1)$ independent of $X$.
In this case, the pair $(Y,X)$ satisfies the model (\ref{eq: QR model}) with $\beta (\tau) = (\tau^{3}/3,-(\tau-1)^{2})^{T}$ and so $Q_{x}(\tau) = \tau^{3}/3 - (\tau-1)^{2} x_2$. Since $Q_{x}'(\tau) = \tau^{2} - 2(\tau-1)x_2$ is minimized at $\tau = x_2$,  the modal function $m(x) = Q_{x}(x_2) = -2x_{2}^{3}/3+2x_{2}^{2} - x_{2}$ is nonlinear. 
\end{remark}

\begin{remark}[Case with no regressors]
\label{rem: Chernoff}
In the simple case where there are no regressors, i.e., $X=1$, our estimator of the mode reduces to $\hat{m} = \hat{Q}(\hat{\tau})$, where $\hat{Q}(\tau) = \hat{F}^{-1}(\tau) = \inf \{ y : \hat{F}(y) \ge \tau \}$ is the empirical quantile function (with $\hat{F}(y) = n^{-1} \sum_{i=1}^{n} I(Y_{i} \le y)$ being the empirical distribution function) and 
\[
\hat{\tau} = \argmin_{\tau} \frac{\hat{Q}(\tau + h) - \hat{Q}(\tau - h)}{2h}.
\]
Our estimator $\hat{m}$ can also be  described by using order statistics $Y_{(1)} \le \cdots \le Y_{(n)}$.  
Since in general $\hat{Q}(\tau) = Y_{(\lceil n\tau \rceil)}$ where $\lceil \cdot \rceil$ is the ceiling function, our estimator $\hat{m}$ coincides with the order statistic $Y_{(\lceil n\hat{\tau} \rceil)}$ where $\hat{\tau}$  minimizes the spacing $Y_{(\lceil n(\tau+h) \rceil)} - Y_{(\lceil n(\tau-h) \rceil)}$. 

It is then clear that our estimator is (related to but) markedly different from Chernoff's \cite{Chernoff1964} estimator of  the unconditional mode of $Y$ that is defined by
\[
\hat{m}_{C} = \argmax_{y} \frac{\hat{F}(y + h) - \hat{F}(y-h)}{2h},
\]
namely, $\hat{m}_{C}$ is the point whose local neighborhood contains the most observations. 
\end{remark}

\begin{remark}[Alternative objective function]
The estimator $\hat{s}_{x}(\tau)$ of $s_{x}(\tau)$ contains a deterministic bias of order $h^{2}$ under the conditions stated in the next section. Alternatively, we may estimate $s_{x}(\tau)$ by 
\begin{equation}
\tilde{s}_{x}(\tau) = \frac{\frac{2}{3} \{ \hat{Q}_{x}(\tau+h) - \hat{Q}_{x}(\tau-h) \} - \frac{1}{12}\{\hat{Q}_{x}(\tau + 2h) - \hat{Q}_{x}(\tau-2h) \}}{h},
\label{eq: alternative objective function}
\end{equation}
which has a bias of order $h^{4}$ under additional smoothness conditions; cf. \cite{BCK2016}. In the present paper, however, we shall use a simpler objective function $\hat{s}_{x}(\tau)$.
\end{remark}

\begin{remark}[Implementation detail]
In the finite sample, $[\tau-h,\tau+h]$ may not be included in $(0,1)$ for some $\tau \in [\varepsilon,1-\varepsilon]$. To fix this, we suggest the following simple modification. Suppose that we compute $\hat{Q}_{x}(\tau)$ on  $[\tau_{\min},\tau_{\max}] \supset [\varepsilon,1-\varepsilon]$; then in practice we suggest to replace $\hat{s}_{x}(\tau)$ by 
\[
\hat{s}_{x}(\tau) = \frac{\hat{Q}_{x}(\tau + \min \{ h, \tau_{\max}-\tau \}) - \hat{Q}_{x}(\tau - \min\{ h,\tau - \tau_{\min} \})}{ \min \{ h, \tau_{\max}-\tau \} + \min \{  h, \tau - \tau_{\min} \}},
\]
which asymptotically coincides with the original definition of $\hat{s}_{x}(\tau)$ uniformly in $\tau \in [\varepsilon,1-\varepsilon]$ (as long as $(\tau_{\min},\tau_{\max}) \supset [\varepsilon,1-\varepsilon]$).
\end{remark}

\begin{remark}[Alternative specifications to the conditional quantile function]
In the present paper, we assume that the conditional quantile function is linear in $X$. 
The linear quantile regression model is the most fundamental modeling in conditional quantile estimation, and is computationally attractive since the computation of the Koenker-Bassett \cite{KoenkerBassett1978} estimate can be formulated as a linear programming problem. Indeed, the computational attractiveness is one of the main motivations to study the proposed estimator of the conditional mode. 

Having said that, we could use alternative specifications to the conditional quantile function to estimate the conditional mode. One possible alternative is a nonlinear quantile regression model $Q_{x}(\tau) = g(x, \beta(\tau))$ where $g$ is some known smooth function (the dimensions of $x$ and $\beta(\tau)$ need not be matched); see e.g. Section 4.4 of \cite{Koenker2005}. 
In this case, we can estimate $\beta(\tau)$ by 
\[
\hat{\beta}(\tau) = \argmin_{\beta} \sum_{i=1}^{n} \rho_{\tau} (Y_{i} - g(X_{i},\beta)),
\]
and thus can estimate $s_{x}(\tau) = Q_{x}'(\tau)$ by $\hat{s}_{x}(\tau) = \{ \hat{Q}_{x}(\tau+h) - \hat{Q}_{x}(\tau-h) \}/(2h)$ with $\hat{Q}_{x}(\tau) = g(x,\hat{\beta}(\tau))$. Alternatively, we can use the expression
\[
s_{x}(\tau) = \Big [ \underbrace{\frac{\partial g (x,\beta)}{\partial \beta}|_{\beta = \beta(\tau)}}_{=: g_{\beta}(x,\beta(\tau))} \Big ]^{T} \frac{d \beta(\tau)}{d \tau},
\]
and estimate $s_{x}(\tau)$ by $\hat{s}_{x}(\tau) = g_{\beta}(x,\hat{\beta}(\tau))^{T} \{ \hat{\beta}(\tau+h) - \hat{\beta}(\tau-h)\}/(2h)$.
It is known that under regularity conditions, similar asymptotic properties to those of the linear quantile regression estimator hold for the nonlinear case (cf. Section 4.4 of \cite{Koenker2005}), and hence it is natural to expect that  asymptotic results analogous to those developed in the next section can be extended to the resulting conditional mode estimator under the nonlinear quantile regression model. 

A yet alternative specification would be a semiparametric single index model $Q_{x}(\tau) = \psi (x^{T}\beta(\tau))$ where $\psi$ is some unknown function.  For given estimators $\hat{\psi}$ and $\hat{\beta}(\tau)$ of $\psi$ and  $\hat{\beta}(\tau)$, we can estimate $Q_{x}(\tau)$ and $s_{x}(\tau)$ by $\hat{Q}_{x}(\tau) = \hat{\psi}(x^{T}\hat{\beta}(\tau))$ and $\hat{s}_{x}(\tau) = \{ \hat{Q}_{x}(\tau + h) - \hat{Q}_{x}(\tau-h)\}/(2h)$, respectively. Alternatively, we can use the expression $s_{x}(\tau) = \psi'(x^{T}\beta(\tau)) d(x^{T}\beta(\tau))/d\tau$ and estimate $d(x^{T}\beta(\tau))/d\tau$ by a difference quotient. Methods to estimate the parametric and nonparametric components in the single index quantile regression model can be found in e.g. \cite{Chaudhuri1997, WuYuYu2010}, and \cite{MaHe2016}. In the single index case, the nonparametric estimation of the link function $\psi$ is involved, whose effect has to be taken into account  when considering asymptotic properties of the resulting conditional mode estimator, which would be a nontrivial challenge. 
\end{remark}

\section{Limiting distributions}
\label{sec: limiting distributions}

\subsection{Limiting distributions}

In this section, we derive  limiting distributions of $\hat{\tau}_{x}$ and $\hat{m}(x)$. To this end, we make the following assumption. Let $\mathcal{X}$ denote the support of $X$.

\begin{assumption}
\label{assumption}
In addition to the baseline assumption stated in the previous section, we assume the following conditions.
\begin{enumerate}
\item[(i)]$\E[ X_{j}^{4}] < \infty$ for all $j=1,\dots,d$. 
\item[(ii)] The matrix $\E[ XX^{T} ]$ is positive definite. 
\item[(iii)] The conditional density $f(y \mid x)$ is three times continuously differentiable with respect to $y$ for each $x \in \mathcal{X}$. Let $f^{(j)}(y \mid x) = \partial^{j}f(y \mid x)/\partial y^{j}$ for $j=0,1,2,3$, where $f^{(0)}(y \mid x) =f(y \mid x)$. There exists a constant $C$ such that $|f^{(j)}(y \mid x)| \le C$ for all $(y,x) \in \R \times \mathcal{X}$ and $j=0,1,2,3$.
\item[(iv)] There exists a positive constant $c$ (that may depend on $\varepsilon$) such that $f(y \mid x) \ge c$ for all $y \in [Q_{x}(\varepsilon/2),Q_{x}(1-\varepsilon/2)]$ and $x \in \mathcal{X}$.
\item[(v)] As $n \to \infty$, $nh^{8} \to 0$ and $nh^{5} \to \infty$.
\end{enumerate}
\end{assumption}

Conditions (i)--(iv) are more or less standard in the quantile regression literature; cf. \cite{Koenker2005}. 
In particular, they require no moment conditions on $Y$. For instance, they allow $\E[|Y|] = \infty$. 
Conditions (iii) and (iv) allow $Q_{x}(\tau)$ to be four times continuously differentiable on $(\varepsilon/2,1-\varepsilon/2)$ with 
\[
\begin{split}
&s_{x}(\tau) := Q_{x}'(\tau) =\frac{1}{f(Q_{x}(\tau) \mid x)}, \ s_{x}'(\tau) = \frac{-f^{(1)}(Q_{x}(\tau) \mid x)}{f(Q_{x}(\tau) \mid x)^{3}}, \\
&s_{x}''(\tau) =  \frac{3 f^{(1)}(Q_{x}(\tau) \mid x)^{2} -f(Q_{x}(\tau) \mid x) f^{(2)}(Q_{x}(\tau) \mid x)}{f(Q_{x}(\tau) \mid x)^{5}}.
\end{split}
\]
Condition (v) is concerned with the bandwidth. The condition $nh^{8} \to 0$ is an ``undersmoothing'' condition. The proof of Theorem \ref{thm: limiting distributions} shows that the estimator $\hat{m}(x)$ contains a deterministic bias of order $h^{2}$, while the stochastic error decreases at rate $(nh^{2})^{-1/3}$. To guarantee that $h^{2} = o((nh^{2})^{-1/3})$, we need $nh^{8} \to 0$. 

Let $\{ B(t) : t \in \R \}$ be a two-sided standard Brownian motion, i.e., a centered Gaussian process with continuous sample paths and covariance function 
\[
\E[B(t_1)B(t_2)] = \begin{cases}
t_{1} & \text{if} \ 0 \le t_{1} \le t_{2} \\
-t_{2} & \text{if} \ t_{1} \le t_{2} \le 0 \\
0 & \text{if} \ t_{1} \le 0 \le t_{2}
\end{cases}
.
\]
Such a two-sided standard Brownian motion can be constructed by generating independent standard Brownian motions $\{ W_{1}(t) : t \ge 0 \}$ and $\{ W_{2}(t) : t \ge 0 \}$, and then defining $B(t) = W_{1}(t)$ for $t \ge 0$ and $B(t) = W_{2}(-t)$ for $t < 0$. In addition, let 
\[
Z = \argmax_{t \in \R} \{ B(t) - t^{2} \},
\]
which exists and is unique almost surely by Lemmas 2.5 and 2.6 in \cite{KimPollard1990}.
The distribution of $Z$ is called \textit{Chernoff's distribution} \cite{Chernoff1964}, and numerical values of quantiles of Chernoff's distribution can be found in \cite{GroeneboomWellner2001}. 

Finally, define the matrix 
\[
J(\tau) = \E[f (X^{T}\beta(\tau) \mid X) XX^{T}],
\]
which is positive definite for every $\tau \in [\varepsilon/2,1-\varepsilon/2]$ under our assumption. 
We are now in position to state the main theorem of this paper. 

\begin{theorem}[Limiting distributions]
\label{thm: limiting distributions}
Pick any $x \in \mathcal{X}$.
Suppose that Assumption \ref{assumption} holds, and in addition that $f^{(2)}(m(x) \mid x) < 0$ and $m(x) \in (Q_{x}(\varepsilon),Q_{x}(1-\varepsilon))$ (or equivalently $\tau_{x} \in (\varepsilon,1-\varepsilon)$). 
Then we have
\begin{equation}
(nh^{2})^{1/3} (\hat{\tau}_{x} - \tau_{x}) \stackrel{d}{\to} (\sigma_{x}/v_{x})^{2/3} Z
\label{eq: first result}
\end{equation}
as $n \to \infty$, where $\sigma_{x}^{2} = \E[(x^{T}J(\tau_{x})^{-1} X)^{2}]/2, v_{x} = s_{x}''(\tau_{x})/2 = -f^{(2)}(m(x) \mid x)/\{ 2 f(m(x) \mid x)^{4} \} \ (> 0)$, and $\sigma_{x} = \sqrt{\sigma_{x}^{2}}$. In addition, we have
\begin{equation}
(nh^{2})^{1/3} (\hat{m}(x) - m(x)) \stackrel{d}{\to} s_{x}(\tau_{x})(\sigma_{x}/v_{x})^{2/3} Z.
\label{eq: second result}
\end{equation}
\end{theorem}

\begin{remark}[Rates of convergence]
The rate of convergence of $\hat{m}(x)$ toward $m(x)$ is $(nh^{2})^{-1/3}$ and can be arbitrarily close to $n^{-1/4}$ under Condition (v), which is independent of the dimension $d$ of the regressor vector. 
The $n^{-1/4}$ rate is likely to be suboptimal from a minimax point of view since it is known that when $X=1$, the minimax rate of estimating the mode under three time differentiability of the underlying density is $n^{-2/7}$; see Theorem 3.1 in \cite{Romano1988}. 
It is worth noting that if we use the alternative objective function $\tilde{s}_{x}(\tau)$ in (\ref{eq: alternative objective function}), the bias of the resulting estimator $\hat{m}(x)$ is reduced to $O(h^{4})$ (this however requires additional smoothness conditions on the conditional density), and therefore the rate of convergence can be arbitrarily close to $n^{-2/7}$. 
\end{remark}

\begin{remark}[Case with no regressors]
In the simple case where there are no regressors, i.e., $X=1$, the limiting distribution of our estimator $\hat{m}$ is as follows. Let $f$ denote the density of $Y$ with mode $m$; then 
$(nh^{2})^{1/3} (\hat{m} - m) \stackrel{d}{\to} \{ 2f(m)^{3}/f''(m)^{2} \}^{1/3} Z$.
In contrast, the limiting distribution of Chernoff's mode estimator (see Remark \ref{rem: Chernoff}) is $(nh^{2})^{1/3} (\hat{m}_{C} - m) \stackrel{d}{\to} \{ 2f(m)/f''(m)^{2} \}^{1/3} Z$, which is slightly different from our limiting distribution. It is worth mentioning that Chernoff's derivation of the preceding limiting distribution in \cite{Chernoff1964} is only heuristic, but can be made rigorous (under regularity conditions) by mimicking the proof of Theorem \ref{thm: limiting distributions}. 
\end{remark}

Interestingly, despite the presence of regressors, the limiting distribution of our estimator $\hat{m}(x)$ is a scale transformation of Chernoff's distribution, which is in contrast to e.g. Manski's  maximum score \cite{Manski1975} whose limiting distribution is given by a maximizer of a Gaussian process  whose covariance function  depends on the distribution of regressors; see Example 6.4 in \cite{KimPollard1990}. 
The fact that the limiting distribution is a scale transformation of Chernoff's distribution makes inference for our estimator $\hat{m}(x)$  relatively simple. 
Namely, an asymptotic confidence interval can be constructed by consistently estimating the constant $s_{x}(\tau_{x})(\sigma_{x}/v_{x})^{2/3}$, which will be discussed in the next section.

The main part of the proof is the proof of the first result (\ref{eq: first result}). 
The second result (\ref{eq: second result}) follows from the $\sqrt{n}$-uniform consistency of the quantile regression estimator and the delta method. 
To prove the first result (\ref{eq: first result}), we begin with expanding the objective function $\hat{s}_{x}(\tau)$ and showing that $\hat{\tau}_{x}$ is an approximate minimizer of the sample average of kernel functions with a uniform kernel; see (\ref{eq: optimality condition}) in the proof.
Since those kernel functions depend on the sample size $n$ via the bandwidth $h = h_n$, 
the result (\ref{eq: first result}) does not follow from the general theorem, Theorem 1.1, in  \cite{KimPollard1990}, which is a pioneering work on cube root asymptotic theory. Theorem 1.1 in \cite{KimPollard1990} covers the case where the objective function is the sample average of functions that do not depend on $n$ and the estimator is $n^{1/3}$-consistent, but its proof does not carry over to our case (cf. the second paragraph in page 192 of \cite{KimPollard1990}). The recent work of \cite{SeoOtsu2018} extends the results of \cite{KimPollard1990} to allow the objective function to depend on the bandwidth (and the data to be dependent), but some of their assumptions are severely restrictive or difficult to verify in our problem. Specifically, Assumption M (i) in \cite{SeoOtsu2018} requires $h_n f_{n,\theta}$ (in their notation) to be uniformly bounded, which in our problem requires the regressor vector $X$ to be bounded (recall that we only assume each coordinate of $X$ to have finite fourth moment); and we (the authors) found that Assumption M (ii) is difficult to verify in our problem. 
Hence, instead of checking the assumptions of \cite{SeoOtsu2018}, we provide a separate and self-contained proof of the result (\ref{eq: first result}), which requires a substantial work.  
Specifically, we show that the ``rescaled'' objective function for which the rescaled estimator $\hat{t} = (nh^{2})^{1/3}(\hat{\tau}_{x}-\tau_{x})$ is an approximate maximizer converges weakly to the process $\{ \sigma_{x} B(t) - v_{x}t^{2} : t \in \R \}$ in the space of locally bounded functions on $\R$, and apply Theorem 2.7 in \cite{KimPollard1990} to conclude that the approximate maximizer $\hat{t} = (nh^{2})^{1/3}(\hat{\tau}_{x}-\tau_{x})$ converges weakly to $\argmax_{t \in \R} \{ \sigma_{x} B(t) - v_{x}t^{2} \}$, which is shown to be equal in distribution to $(\sigma_{x}/v_{x})^{2/3} Z$; see Step 5 of the proof.

Next, we consider a multivariate limit theorem for the proposed estimator. Let $x^{1},\dots,x^{L} \in \mathcal{X}$ be a finite number of design points with $L$ independent of $n$, and let 
\[
\tau_{(1)} > \tau_{(2)} > \cdots > \tau_{(M)}
\]
denote the distinct values of  $\tau_{x^{1}},\dots,\tau_{x^{L}}$. Set $S_{k} = \{ j \in \{ 1,\dots, L \} : \tau_{x^{j}} = \tau_{(k)} \} $ with $s_{k} = \Card (S_{k})$ for $k=1,\dots,M$. For each $k=1,\dots,M$, let $\{ \mathbb{B}_{k}((t_{j})_{j \in S_{k}}) : (t_{j})_{j \in S_{k}} \in \R^{s_{k}} \}$ denote a centered Gaussian process with covariance function 
\[
\Cov \left ( \mathbb{B}_{k}((t_{i})_{i \in S_{k}}),\mathbb{B}_{k}((t_{j}')_{j \in S_{k}}) \right) = \frac{1}{2} \sum_{i,j \in S_{k}} (x^{i})^{T} J(\tau_{(k)})^{-1} \E[XX^{T}] J(\tau_{(k)})^{-1} x^{j} \E[B(t_{i})B(t_{j}')].
\]
We note that the construction of the  $\mathbb{B}_{k}$-process depends on the design points $x^{1}.\dots,x^{L}$. 
Recall that a version of a stochastic process is another process with the same finite dimensional distributions. 

\begin{corollary}
\label{cor: multivariate}
Suppose that Assumption \ref{assumption} holds, and in addition that $f^{(2)}(m(x) \mid x) < 0$ and $m(x) \in (Q_{x}(\varepsilon),Q_{x}(1-\varepsilon))$ for all $x \in \{ x^{1},\dots,x^{L} \}$. Then, for each $k=1,\dots,M$, there exists a version of the $\mathbb{B}_{k}$-process with continuous paths, and denoting the continuous version by the same symbol $\mathbb{B}_{k}$, we have
\[
(nh^{2})^{1/3} ( \hat{\tau}_{x^{1}},\dots,\hat{\tau}_{x^{L}})^{T} \stackrel{d}{\to} (W_{1},\dots,W_{L})^{T}
\]
as $n \to \infty$, where $(W_{j})_{j \in S_{k}}, k=1,\dots,M$ are independent, and for each $k=1,\dots,M$, 
\[
(W_{j})_{j \in S_{k}} \stackrel{d}{=} \argmax_{(t_{j})_{j \in S_{k}} \in \R^{s_{k}}} \left \{ \mathbb{B}_{k}((t_{j})_{j \in S_{k}})-\sum_{j \in S_{k}} v_{x^{j}} t_{j}^{2} \right \}.
\]
In addition, we have
\[
(nh^{2})^{1/3} \left ( \hat{m}(x^{1}) - m(x^{1}),\dots,\hat{m}(x^{L}) - m(x^{L}) \right )^{T}  \stackrel{d}{\to} \left(s_{x^{1}}(\tau_{x^{1}})W_{1},\dots,s_{x^{L}}(\tau_{x^{L}}) W_{L} \right)^{T}.
\]
\end{corollary}

In the special case when $\tau_{x^{1}},\dots,\tau_{x^{L}}$ are all distinct, we have
\[
\begin{split}
&(nh^{2})^{1/3} \left ( \hat{m}(x^{1}) - m(x^{1}),\dots,\hat{m}(x^{L}) - m(x^{L}) \right )^{T} \\
&\quad \stackrel{d}{\to} \left(s_{x^{1}}(\tau_{x^{1}})(\sigma_{x^{1}}/v_{x^{1}})^{2/3} Z_{1},\dots,s_{x^{L}}(\tau_{x^{L}})(\sigma_{x^{L}}/v_{x^{L}})^{2/3} Z_{L} \right)^{T},
\end{split}
\]
where $Z_{1},\dots,Z_{L}$ are independent Chernoff random variables. 

Corollary \ref{cor: multivariate} implies that 
\[
(nh^{2})^{1/3} \max_{1 \le j \le L} | \hat{m}(x^{j}) - m(x^{j}) | \stackrel{d}{\to} \max_{1 \le j \le L} | s_{x^{j}}(\tau_{x^{j}}) W_{j}|,
\]
which can be used to construct simultaneous confidence intervals for $m(x)$ over the design points $x^{1},\dots,x^{L}$; see Remark \ref{rem: simultaneous interval} ahead. 

\begin{remark}[Uniform rate over expanding sets of design points] 
\label{rem: uniform rate}
It is of interest to study the rate of convergence and limiting distribution of the $L^{\infty}$-distance between the proposed estimator and the true modal function on a continuum set of design points or expanding sets of design points, since e.g. such limiting distribution enables us to construct simultaneous confidence bands. To the best of our knowledge, however, much less is known about the rate of convergence and (especially) limiting distribution for the $L^{\infty}$-distance in nonstandard nonparametric estimation problems than standard nonparametric estimation problems with Gaussian limits, and we believe that the problem is challenging. One exception is the work of \cite{Durot2012}, which derives the uniform rate of convergence and the limiting distribution of  the $L^{\infty}$-distance for the Grenander \cite{Grenander1956} estimator (precisely speaking \cite{Durot2012} cover more general Grenander-\textit{type} estimators); see also the recent review article by \cite{Durot2018}. Their argument depends substantially on the specific construction of the Grenander estimator and can not be directly extended to our estimator. It is thus beyond the scope of this paper to formally study the uniform rate and the limiting distribution of the $L^{\infty}$-distance to our estimator, but we will give some heuristic discussion on this question, which we believe is of some interest to the reader. 

To simplify the question, we confine ourselves to the maximum distance on expanding sets of design points $x^{1},\dots,x^{L}$ with $L=L_{n} \to \infty$. Suppose in addition that $\tau_{x^{1}},\dots,\tau_{x^{L_{n}}}$ are all distinct.
Then by Corollary \ref{cor: multivariate} it is expected that  $\max_{1 \le j \le L_{n}} (nh^{2})^{1/3}|\hat{m}(x^{j}) - m(x^{j})|/\{ s_{x^{j}}(\tau_{x^{j}})(\sigma_{x^{j}}/v_{x^{j}})^{2/3}\}$ could be approximated by $\max_{1 \le j \le L_{n}} | Z_{j} | =: |Z|_{(L_{n})}$ as long as $L_{n} \to \infty$ sufficiently slowly. In Appendix \ref{sec: Gumbel}, we will show that, for the norming constants 
\[
\begin{split}
a_{L_{n}} &=  3 \left ( \frac{2}{3} \right )^{1/3} (\log L_{n})^{2/3}, \\
b_{L_{n}}' &=  \left(\frac{3}{2}\log L_{n} \right) ^{1/3} - \frac{1}{a_{L_{n}}} \left [\kappa \left ( \frac{3}{2} \log L_{n} \right)^{1/3} + \frac{1}{3} \log \log L_{n} + \frac{1}{3} \log \frac{3}{2} - \log (2\lambda)  \right ],
\end{split}
\]
where $\lambda$ and $\kappa$ are positive constants (see Appendix \ref{sec: Gumbel}), we have 
\[
a_{L_{n}} (|Z|_{(L_{n})} - b_{L_{n}}') \stackrel{d}{\to} \underbrace{\Lambda}_{\text{Gumbel distribution}}.
\]
In particular, $|Z|_{(L_{n})} = b_{L_{n}}'/a_{L_{n}} + O_{\Prob}(1/a_{L_{n}}) = O_{\Prob} ((\log L_{n})^{1/3})$, and as long as $L_{n}$ grows at most polynomially fast in $n$, $|Z|_{(L_{n})} = O_{\Prob} ((\log n)^{1/3})$. This suggests that the uniform rate of the proposed estimator would be $O_{\Prob} ((nh^{2}/\log L_{n})^{-1/3})$ and the maximum distance would converge in distribution to the Gumbel distribution after normalization. The preceding argument is heuristic since Corollary \ref{cor: multivariate} only holds with fixed $L$ (and extending the corollary to the case where $L = L_{n} \to \infty$ is a substantial technical challenge), and the rigorous result is left to future research. 
\end{remark}

\subsection{Inference} 
\label{subsec: inference}

\subsubsection{Analytical confidence intervals}
Theorem \ref{thm: limiting distributions} allows us to construct pointwise confidence intervals for $m(x)$ by consistently estimating the nuisance parameters $\sigma_{x}^{2},v_{x}$, and $s_{x}(\tau_{x})$.  

The parameter $s_{x}(\tau_{x})$ can be estimated by $\hat{s}_{x}(\hat{\tau}_{x})$. 
Next, consider to  estimate $\sigma_{x}^{2}$. For the notational convenience, let $\Sigma = \E[XX^{T}]$ and so  $\sigma_{x}^{2} =x^{T} J(\tau_{x})^{-1} \Sigma J(\tau_{x})^{-1} x/2$.
The matrices  $\Sigma$ and $J(\tau)$ can be estimated by 
\[
\hat{\Sigma} = \frac{1}{n}\sum_{i=1}^{n}X_{i}X_{i}^{T} \quad \text{and} \quad 
\hat{J}(\tau) = \frac{1}{2nh} \sum_{i=1}^{n} I(|Y_{i} - X_{i}^{T}\hat{\beta}(\tau)| \le h) X_{i}X_{i}^{T},
\]
respectively, so that we can estimate $\sigma_{x}^{2}$ by 
\[
\hat{\sigma}_{x}^{2} = \frac{1}{2}x^{T}\hat{J}(\hat{\tau}_{x})^{-1}\hat{\Sigma}\hat{J}(\hat{\tau}_{x})^{-1}x,
\]
where $\hat{J}(\tau)$ is Powell's kernel estimator \cite{Powell1986}. 
Finally, consider to estimate $v_{x} = s_{x}''(\tau_{x})/2$. To this end, we estimate $s_{x}''(\tau) = Q_{x}'''(\tau)$ by a numerical differentiation of $\hat{Q}_{x}(\tau)$. 
Namely, define the operator $\Delta_{h}$ by $\Delta_{h} g(\tau) = (g(\tau+h)-g(\tau-h))/(2h)$, and $\Delta_{h}^{j} g = \Delta_{h} (\Delta_{h}^{j-1} g)$ recursively for $j=2,3,\dots$. 
Then we can estimate $v_{x}$ by 
\begin{equation}
\hat{v}_{x} = \frac{1}{2}\Delta_{h}^{3} \hat{Q}_{x}(\hat{\tau}_{x}). 
\label{eq: third derivative}
\end{equation}
The bandwidths used in $\hat{J}(\tau)$ and $\hat{v}_{x}$ can be different from that for $\hat{\tau}_{x}$. 
See Remark \ref{rem: estimation of v} ahead for alternative estimators for $v_{x}$. 
The following proposition shows that these estimators are indeed consistent under almost the  same conditions as in Theorem \ref{thm: limiting distributions}.

\begin{proposition}[Consistency of estimators for nuisance parameters]
\label{prop: inference}
Suppose that the conditions of Theorem \ref{thm: limiting distributions} hold and  in addition that $nh^{5}/\log n \to \infty$. Then we have 
$\hat{\sigma}_{x}^{2} \stackrel{\Prob}{\to} \sigma_{x}^{2}, \ \hat{v}_{x} \stackrel{\Prob}{\to} v_{x}$, and $\hat{s}_{x}(\hat{\tau}_{x}) \stackrel{\Prob}{\to} s_{x}(\tau_{x})$
as $n \to \infty$.
\end{proposition}

Now, since Chernoff's distribution is symmetric about the origin, an asymptotic $(1-\alpha)$-confidence interval for $m(x)$ is given by 
\[
\left [ \hat{m}(x) \pm \frac{\hat{s}_{x}(\hat{\tau}_{x}) (\hat{\sigma}_{x}/\hat{v}_{x})^{2/3}}{(nh^{2})^{1/3}} q_{1-\alpha/2} \right],
\]
where $q_{1-\alpha/2}$ is the $(1-\alpha/2)$-quantile of Chernoff's distribution. 
For example, Table 2 in \cite{GroeneboomWellner2001} yields that $q_{0.975} \approx 0.998181$. 

\begin{remark}[Alternative estimators for $v_{x}$]
\label{rem: estimation of v}
Alternatively to the estimator $\hat{v}_{x}$, we may use 
\[
\tilde{v}_{x} = \frac{1}{2h^{3}} \left (  \hat{Q}_{x}(\hat{\tau}_{x}+2h) - \hat{Q}_{x}(\hat{\tau}_{x} -2h) - 2\{\hat{Q}_{x}(\hat{\tau}_{x}+h) - \hat{Q}_{x}(\hat{\tau}_{x} -h) \} \right ),
\]
which is consistent under additional smoothness conditions on the conditional density. 

Still, higher order numerical differentials tend to be unstable in the finite sample. 
Instead, we may use the expression $v_{x} = -f^{(2)}(m(x) \mid x)s(\tau_{x})^{4}/2$, and estimate $f^{(2)}(m(x) \mid x)$ by a kernel method. 
Suppose that $X$ is decomposed as $X= (X^{C},X^{D})$ where $X^{C} \in \R^{d_{C}}$ is continuous and $X^{D} \in \R^{d-d_{C}}$ is discrete.  Let $\mathbb{K}_1: \R \to \R$ and $\mathbb{K}_2: \R^{d_{C}} \to \R$ be kernel functions (i.e., functions that integrate to $1$) where $\mathbb{K}_1$ is twice differentiable. 
For  given bandwidths $b_{X} =b_{X,n} \to 0$ and $b_{Y} = b_{Y,n} \to 0$, we may estimate $f^{(2)}(m(x) \mid x)$ with $x = (x^{C},x^{D})$ by 
\[
\hat{f}^{(2)} (\hat{m}(x) \mid x) = \frac{(nb_{Y}^{3}b_{X}^{d_{C}})^{-1}\sum_{i=1}^{n} \mathbb{K}_1''((\hat{m}(x)-Y_{i})/b_{Y}) \mathbb{K}_2((x^{C}-X_{i}^{C})/b_{X})I(X_{i}^{D}=x^{D})}{(nb_{X}^{d_{C}})^{-1} \sum_{i=1}^{n} \mathbb{K}_2((x^{C}-X_{i}^{C})/b_{X})I(X_{i}^{D} = x^{D})},
\]
which is consistent under appropriate conditions. This leads to an alternative estimator for $v_{x}$: 
\begin{equation}
\check{v}_{x} = -\hat{f}^{(2)}(\hat{m}(x) \mid x) \hat{s}(\hat{\tau}_{x})^{4}/2.
\label{eq: kernel estimation}
\end{equation}
In the simulation study, we use the kernel-based estimator $\check{v}_{x}$ for $v_{x}$. 
\end{remark}

\subsubsection{Subsampling}

It is known that the nonparametric bootstrap in general fails to be consistent for $n^{1/3}$-consistent estimators (cf. \cite{AbrevayaHuang2005, LegerMacGibbon2006, Kosorok2008, Sen2010}) and so it is unlikely that the bootstrap would be consistent for our estimator $\hat{m}(x)$. Instead, since the limiting distribution is a scale transformation of Chernoff's distribution that is absolutely continuous, the subsampling 
provides a valid inference method for our estimator $\hat{m}(x)$; see \cite{PolitisRomano1994,Politis1999}. 
Let $\hat{m} (x) = \hat{m}_{n}(x) = \hat{m}_{n}(x;(Y_{1},X_{1}),\dots,(Y_{n},X_{n}))$ and $h=h_n$, and let $W_{1},\dots,W_{N}$ be the $N=\binom{n}{\ell}$ subsets of $\{ (Y_{1},X_{1}),\dots,(Y_{n},X_{n}) \}$ of size $\ell \ (<n)$. Consider the subsampling distribution
\begin{equation}
U_{n,\ell}(x;t) = \frac{1}{N} \sum_{i=1}^{N} I\left \{ (\ell h_{\ell}^{2})^{1/3}(\hat{m}_{\ell}(x;W_{i}) - \hat{m}_{n}(x)) \le t \right \}.
\label{eq: subsampling distribution}
\end{equation}
Then, under the same conditions as in Theorem \ref{thm: limiting distributions}, we have 
\[
\sup_{t \in \R} \left | U_{n,\ell} (x;t) - \Prob \left ( s_{x}(\tau_{x})(\sigma_{x}/v_{x})^{2/3}Z \le t \right ) \right| \stackrel{\Prob}{\to} 0,
\]
provided that $\ell=\ell_n \to \infty$ and $\ell = o(n)$. Hence, denoting by $\hat{q}_{n,\ell}(x;1-\alpha)$ the $(1-\alpha)$-quantile of $U_{n,\ell}(x;\cdot)$, i.e.,
\[
\hat{q}_{n,\ell}(x;1-\alpha) = \inf \{ t : U_{n,\ell}(x;t) \ge 1-\alpha \},
\]
an asymptotic $(1-\alpha)$-confidence interval for $m(x)$ is given by
\[
\left [ \hat{m}_{n}(x) - \frac{\hat{q}_{n,\ell}(x;1-\alpha/2)}{(nh_{n}^{2})^{1/3}}, \hat{m}_{n}(x) - \frac{\hat{q}_{n,\ell}(x;\alpha/2)}{(nh_{n}^{2})^{1/3}}\right].
\]
Some comments on the subsampling confidence interval are in order.
\begin{remark}[Comments on subsampling confidence interval]
(i) In practice, $N = \binom{n}{\ell}$ is too large and so the computation of the complete average over $i=1,\dots,N$ in (\ref{eq: subsampling distribution}) is too demanding. Instead, we can take the average of a randomly selected subset of $\{ 1,\dots, N \}$; see Corollary 2.4.1 in \cite{Politis1999}. 

(ii) The bandwidth $h_{\ell}$ used in each subsample may be taken as $h_{n}$ as long as $\ell_{n}h_n^{8} \to 0$ and $\ell_{n}h_n^{5} \to \infty$. 
\end{remark}

\begin{remark}[Simultaneous confidence intervals over finite design points]
\label{rem: simultaneous interval}
Consider the setting of Corollary \ref{cor: multivariate}, and let $\nu_{1-\alpha}$ denote the $(1-\alpha)$ quantile of $\max_{1 \le j \le L} |s_{x^{j}}(\tau^{j}) W_{j}|$. Then a simultaneous confidence interval for $m(x)$ over the design points $x^{1},\dots,x^{L}$ is given by 
\[
\left [ \hat{m}(x^{j}) \pm \frac{\nu_{1-\alpha}}{(nh^{2})^{1/3}} \right ], \ j=1,\dots,L. 
\]
In general the distribution of $(W_{1},\dots,W_{L})^{T}$ is complicated as it depends on whether there are ties in $\tau_{x^{1}},\dots,\tau_{x^{L}}$, so analytical estimation of $\nu_{1-\alpha}$ is difficult. 
Instead, we can use the subsampling to estimate $\nu_{1-\alpha}$. The procedure is analogous to the pointwise case and hence omitted. 
\end{remark}

\section{Numerical results}
\label{sec: numerical results}

\subsection{Bandwidth selection}
\label{subsec: bandwidth selection}
The proposed estimator requires to choose the bandwidth $h$.
We suggest here a simple method to choose the bandwidth, which is based on a modification to the bandwidth selection rule suggested in \cite{KoenkerMachado1999}. 
The baseline idea of our approach is to select the bandwidth in such a way that the sparsity function $s_{x}(\tau)$ is well estimated. A similar approach is used in \cite{EinbecTutz2006} who adapt the smoothing bandwidth to kernel estimation of multi-modal regression by optimizing the conditional density estimation rate. The performance of the sparsity function estimate $\hat{s}_{x}(\tau)$ depends on the quantile $\tau$ of interest, and so the constant involved in the bandwidth should adapt to $\tau$. Since we are interested in $s_{x}(\tau)$ around $\tau = \tau_{x}$, we aim at choosing $h$ in such a way that $\hat{s}_{x}(\tau)$ around $\tau = \tau_{x}$ tends to be accurate but modify the rate of $h$ so that it satisfies Condition (v) in Assumption \ref{assumption}. 

For estimation of $s_{x}(\tau)$ based on quantile regression, \cite{KoenkerMachado1999} suggest to use the $\tau$-dependent bandwidth 
\[
h^{KM}(\tau)= n^{-1/3} z_\alpha^{2/3} \left\{1.5 \frac{\phi(\Phi^{-1}(\tau))}{2\Phi^{-1}(\tau)^2 +1} \right\}^{1/3},
\]
where $\phi$ and $\Phi$ are the density and distribution functions of $N(0,1)$, and $z_{\alpha} = \Phi^{-1}(1-\alpha/2)$. 
We set $\alpha = 0.05$. The bandwidth $h^{KM}(\tau)$ does not satisfy Condition (v) in Assumption \ref{assumption} and is $\tau$-dependent, and so we shall modify $h^{KM}(\tau)$ as follows: (i)  pick any design point $x$ in the support of $X$; (ii) use the pilot bandwidth $h^{\mathrm{pilot}} = n^{1/6} h^{KM}(0.5) \propto n^{-1/6}$ to construct a preliminary estimator $\hat{\tau}_{x}^{\mathrm{prelim}}$ of $\tau_{x}$; (iii) and  use $h_n = h_{n,x} = n^{1/6} h^{KM} (\hat{\tau}_x^{\mathrm{prelim}})$ to construct a final estimator $\hat{m}(x)$. The simulation results suggest that, although it would not be optimal, this bandwidth selection rule works reasonably well.

\subsection{Simulation results} 
\subsubsection{Comparison of RMSEs}
We compare the performance of our estimator with that of the linear modal regression estimator of \cite{KempSantos2012,YaoLi2014} via the root mean square error (RMSE) $\sqrt{\E_{X^{*}}[\{ \hat{m}(X^{*}) - m(X^{*}) \}^{2}]}$ where $X^{*} \stackrel{d}{=} X$ is independent of the data and $\E_{X^{*}}$ is the expectation with respect to $X^{*}$. 
We consider two settings: the first one is the case where the modal function is linear while the second one is the case where the modal function is nonlinear. 

Case (i). Consider a linear location-scale model
\[
Y = 1 + X_2 -3X_3 + X_4 + X_2 \nu,
\]
where $X=(1,X_2,X_3,X_4)^{T}, \ X_2, X_3 \sim U(0,1), \ X_4 \sim N(0,1)$, and $\nu \sim Ga(3,0.5)$ (the Gamma distribution with shape parameter $3$ and scale parameter $0.5$).
In this case, both the conditional quantile and modal functions are linear in $X$. In fact, $Q_{\tau}(X) = 1 + (1+F^{-1}(\tau))X_2 - 3X_3 + X_4$, where $F$ denotes the distribution function of $\nu$. 
In addition, since the mode of $Ga(3,0.5)$ is $1$, the modal function is $m(X) = 1+2X_2-3X_3+X_4$.

Case (ii). Consider the following data generating process
\[
Y = U^{3}/3- X_2 (U-1)^{2},
\]
where $X = (1,X_2)^{T}$, $X_2 \sim U(0,1)$, and $U \sim U(0,1)$ independent of $X$. In this case, the conditional quantile function is linear, $Q_{\tau}(X) = \tau^{3}/3 - X_2 (\tau-1)^{2}$, but the modal function is nonlinear, $m(X) = -2X_2^{3}/3+2X_2^{2}-X_2$; see Remark \ref{rem: linearity}. 

In this simulation study, we choose $\varepsilon=0.1$ and compute $\hat{Q}_{x}(\tau)$ for 100 equally spaced grids on $[\tau_{\min},\tau_{\max}] = [0.05,0.95]$. To implement the linear modal regression estimator, we follow the EM algorithm and the bandwidth selection rule suggested in \cite{YaoLi2014}.
The number of Monte Carlo repetitions is $1000$ for each case. 

Figures \ref{fig: RMSE Case1} and \ref{fig: RMSE Case2} present the box plots of RMSEs of the linear modal regression and proposed estimators for Cases (i) and (ii), respectively, with $n=500, 1000$, and $2000$. 
These figures lead to the following observations. First, in both cases, the RMSE of the proposed estimator overall decreases as the sample size increases. Second, the proposed estimator tends to be more variable than the linear modal regression estimator, so that the interquartile range of the RMSE is wider for the proposed estimator than the linear modal regression estimator.  
Third, in both cases, the proposed estimator outperforms the linear modal regression estimator. The superior performance of the proposed estimator in Case (ii) is not surprising since the true modal function is nonlinear in that case and so the linear modal regression estimator is not consistent. 
Interestingly,  even when the true modal function is linear (Case (i)), the proposed estimator performs substantially better than the linear modal regression estimator. This may be partly because the EM algorithm used to compute linear modal regression estimates failed to find global optimal solutions. 
Overall, the figures confirm that the proposed estimator works  well in practice. 
%

\begin{figure}[t!] 
   \centering
       \includegraphics[scale=0.33]{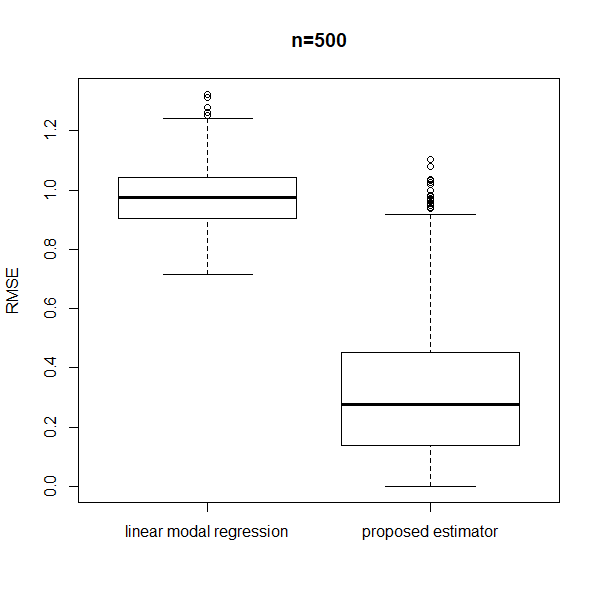}
       \includegraphics[scale=0.33]{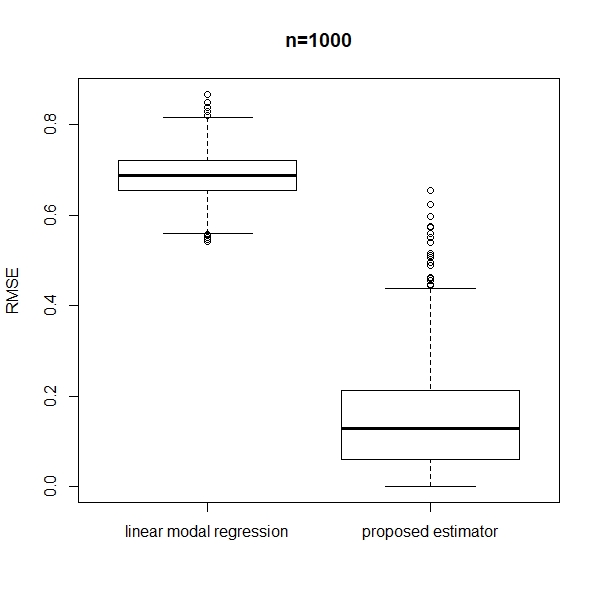}
       \includegraphics[scale=0.33]{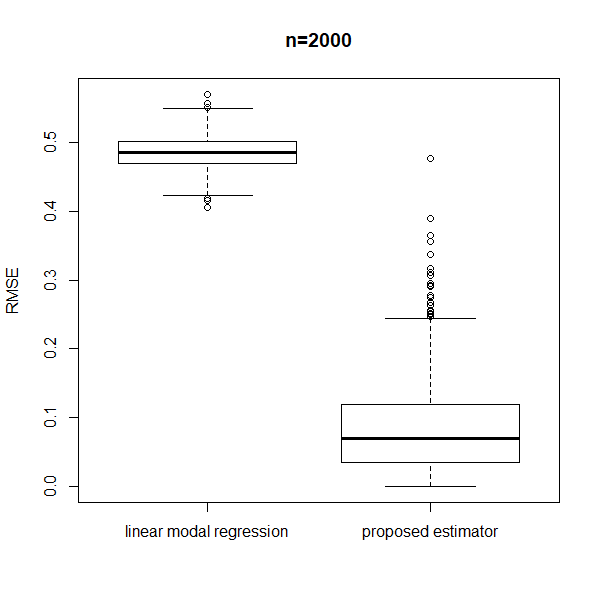}
    \caption{Box plots of RMSEs of the linear modal regression and proposed estimators for Case (i) with $n=500$ (left), $n=1000$ (middle), and $n=2000$ (right). }
    \label{fig: RMSE Case1}
\end{figure}

\begin{figure}[t!]
   \centering
       \includegraphics[scale=0.75]{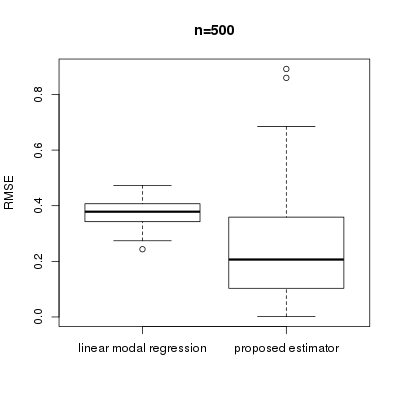}
       \includegraphics[scale=0.75]{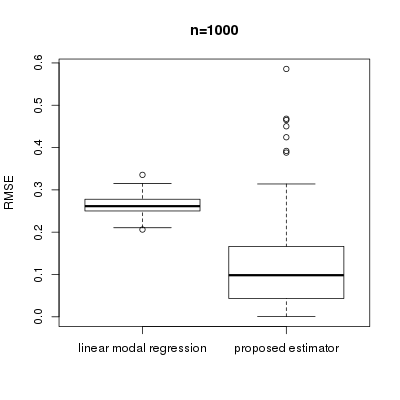}
       \includegraphics[scale=0.75]{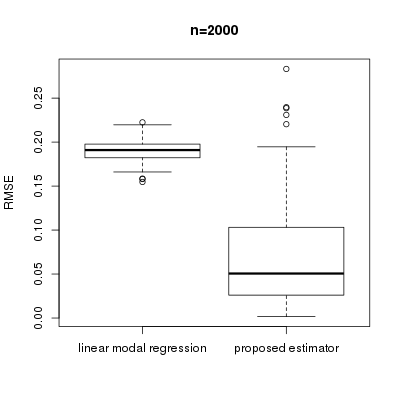}
    \caption{Box plots of RMSEs of the linear modal regression and proposed estimators for Case (ii) with $n=500$ (left), $n=1000$ (middle), and $n=2000$ (right). }
    \label{fig: RMSE Case2}
\end{figure}

\subsubsection{Coverage probabilities of confidence intervals} 

Next, we assess the performance of analytical and subsampling confidence intervals considered in Section \ref{subsec: inference}. 
We follow the data generating process of Case (ii) and evaluate Monte Carlo average and median lengths,  and coverage probabilities of confidence intervals at three design points $x_2 = 0.25, 0.5$, and $0.75$. We consider two nominal coverage probabilities of 99\% and 95\%. To implement the analytical confidence interval, we use the kernel-based estimator $\check{v}_{x}$  given in (\ref{eq: kernel estimation}) for $v_{x}$. To construct $\check{v}_{x}$, we use the Gaussian kernel for $\mathbb{K}_1$ and the Epanechnikov kernel for $\mathbb{K}_2$ together with bandwidths $b_{Y} = n^{-1/9}\hat{\sigma}_{Y}$ and $b_{X} = n^{-1/5} \hat{\sigma}_{X}$ where $\hat{\sigma}_{Y}$ and $\hat{\sigma}_{X}$ are the sample standard deviations of  $Y$ and $X$, respectively. 
To implement the subsampling confidence interval, we examine two subsample sizes: $\ell = 0.1 n$ and $0.2 n$. 
In this simulation study, instead of taking the average of whole subsamples in (\ref{eq: subsampling distribution}), we take the average of 250 randomly chosen subsamples. 
When applying the bandwidth selection rule to the subsample, we use the pilot bandwidth computed using the full sample. 

Tables \ref{table: analytic 99}--\ref{table: subsampling 95} present the simulation results on the confidence intervals. The tables show that both confidence intervals work reasonable well, given that the convergence rate of the estimator is relatively slow. 
It is worth noting that the estimators for the nuisance parameters $s_{x}(\tau_{x})$ and $v_{x}$ tend to be unstable, which results in the discrepancy between the average and median lengths of the analytical confidence interval. The subsample confidence interval is able to avoid estimation of those nuisance parameters, and so the length of the subsampling confidence interval tends to be shorter than that of the analytical confidence interval. In terms of the coverage probability, the subsampling confidence interval with subsample size $0.2n$ works the best. 

{\small

\begin{table}[h]
	
	\begin{tabular}{|c|c|c|c|c|} \hline
		Design point & Sample size & Ave. length &Med. length& Cov. probability \\ \hline \hline
		$x_2=0.25$ & $n=500$ & 0.494 & 0.419&0.981
		\\
		& $n=1000$&0.359 &0.315 &	0.986
		\\
		& $n=2000$ &0.247  &0.220 &0.985 \\ 
		\hline
		
		$x_2=0.50$ & $n=500$ &0.715 &0.599 &1.000\\
		& $n=1000$&0.506 &0.475  &0.997	\\
		& $n=2000$ &0.392  &0.380 & 0.992\\ 
		\hline
		
		$x_2=0.75$ & $n=500$ &1.045 &0.878 &0.978\\
		& $n=1000$&0.724 & 0.653&0.977	\\
		& $n=2000$ &0.524  &0.488 &0.956 \\ 
		\hline
	\end{tabular}
	\medskip
	\caption{Monte Carlo average and median lengths, and coverage probabilities of the $99\%$ analytical confidence interval. }
	\label{table: analytic 99}
\end{table}

\begin{table}[h]
	
	\begin{tabular}{|c|c|c|c|c|} \hline
	Design point & Sample size & Ave. length &Med. length& Cov. probability \\ \hline \hline
	$x_2=0.25$ & $n=500$ &0.309  &0.242 &0.948\\
	& $n=1000$&0.207 &0.175 &0.941	\\
	& $n=2000$ & 0.139 &0.128 &0.952 \\ 
	\hline
	
	$x_2=0.50$ & $n=500$ &0.459 &0.343 &0.987\\
	& $n=1000$& 0.302 & 0.269&0.933	\\
	& $n=2000$ &0.226  &0.221&0.894 \\ 
	\hline
	
	$x_2=0.75$ & $n=500$ &0.660  &0.534 &0.873\\
	& $n=1000$&0.429 &0.371 &0.869	\\
	& $n=2000$ &0.302  &0.278 &0.845 \\ 
	\hline
\end{tabular}
	\medskip
	\caption{Monte Carlo average and median lengths, and coverage probabilities  of the $95\%$ analytical confidence interval.}
\label{table: analytic 95}
	
\end{table}

\begin{table}[h]
	
	\begin{tabular}{|c|c|c|c|c|c|} \hline
		Design point & Sample size & Subsample size & Ave. length  &Med. length& Cov. probability \\ \hline \hline
		$x_2=0.25$ & $n=500$ & $0.1n$&0.232 &0.234 &0.959\\	
		& & $0.2n$&0.250 &0.262 &0.991\\\hline
		
		& $n=1000$  &$0.1n$ &0.208 &0.214 &0.966 \\
		&   &$0.2n$ &0.191&0.184&0.997\\ \hline
		
		& $n=2000$ &$0.1n$ &0.148&0.146&1.000\\
		&  &$0.2n$ &0.146&0.143&1.000\\\hline
		
		$x_2=0.50$ & $n=500$ & $0.1n$&0.336 &0.337&0.946
		\\
		& & $0.2n$&0.405&0.407 &0.999\\\hline
		
		& $n=1000$  &$0.1n$ &0.326&0.327 &0.973\\
		&   &$0.2n$ &0.391& 0.395&0.998\\ \hline
		
		& $n=2000$ &$0.1n$ &0.371&0.382&1.000\\
		&  &$0.2n$ &0.371&0.382&0.999\\\hline
		
		$x_2=0.75$ & $n=500$ & $0.1n$&0.447 &0.450&0.822
		\\
		& & $0.2n$&0.529 &0.538 &0.917 \\\hline
		
		& $n=1000$  &$0.1n$ &0.430&0.433 &0.847\\
		&   &$0.2n$ &0.488&0.508&0.961\\ \hline
		
		& $n=2000$ &$0.1n$ &0.416&0.415&0.971\\
		&  &$0.2n$ &0.423&0.416	&0.971\\\hline
		
	\end{tabular}
	\medskip
	
	\caption{Monte Carlo average and median lengths, and coverage probabilities of the $99\%$ subsampling confidence interval.}

\label{table: subsampling 99}
\end{table}

\begin{table}[h]
	
	\begin{tabular}{|c|c|c|c|c|c|} \hline
		Design point & Sample size & Subsample size & Ave. length  &Med. length& Cov. probability \\ \hline \hline
		$x_2=0.25$ & $n=500$ & $0.1n$&0.203 & 0.208&0.926\\	
		& & $0.2n$&0.198 & 0.195&0.982\\\hline
		
		& $n=1000$  &$0.1n$ &0.166 & 0.166&0.947 \\
		&   &$0.2n$ &0.148&0.145&0.993\\ \hline
		
		& $n=2000$ &$0.1n$ &0.120&0.119&0.997\\
		&  &$0.2n$ &0.118&0.116&0.998\\\hline
		
		$x_2=0.50$ & $n=500$ & $0.1n$&0.313 &0.314&0.899
		\\
		& & $0.2n$&0.374&0.380 &0.989\\\hline
		
		& $n=1000$  &$0.1n$ &0.304&0.306 &0.968\\
		&   &$0.2n$ &0.353&0.366 &0.997\\ \hline
		
		& $n=2000$ &$0.1n$ &0.316&0.326&0.994\\
		&  &$0.2n$ &0.318&0.328&0.996\\\hline
		
		$x_2=0.75$ & $n=500$ & $0.1n$&0.413 &0.416&0.779
		\\
		& & $0.2n$&0.473 &0.490 &0.887 \\\hline
		
		& $n=1000$  &$0.1n$ &0.388&0.396 &0.808\\
		&   &$0.2n$ &0.412&0.415&0.937\\ \hline
		
		& $n=2000$ &$0.1n$ &0.335&0.328&0.958\\
		&  &$0.2n$ &0.342&	0.336&0.959\\\hline
		
	\end{tabular}
	\medskip
	
	\caption{Monte Carlo average and median lengths, and coverage probabilities of the $95\%$ subsampling confidence interval.}

\label{table: subsampling 95}
\end{table}
}

\subsection{Combined Cycle Power Plant Data}

The electric energy output provided by a power plant fluctuates through the year because of several environmental conditions, and prediction of the electricity output given such environmental conditions is of interest. 
We apply the proposed estimator to predicting the net hourly electrical energy output using Combined Cycle Power Plant Data \cite{Kaya2012,Tufekci2014}. 
The data set is taken from \texttt{https://archive.ics.uci.edu/ml/datasets/
Combined+Cycle+Power+Plant} and consists  of 9568 data points collected from a Combined Cycle Power Plant over 6 years (2006-2011). It contains hourly average ambient variables Temperature, Ambient Pressure, Relative Humidity, Exhaust Vacuum, and the net hourly electrical energy output, where the first four variables are regressors and the last variable is a response. 
For this data, the conditional distribution tends to be skewed, and therefore it would be natural to estimate the conditional mode. 
Figure \ref{fig:power plant} shows the estimate of the conditional density given one of the regressors (Exhaust Vacuum).  It is seen that the conditional density estimate is highly skewed and the pattern of the skewness depends on the value of the regressor. 

\begin{figure}[t!] 
   \centering
       \includegraphics[scale=1]{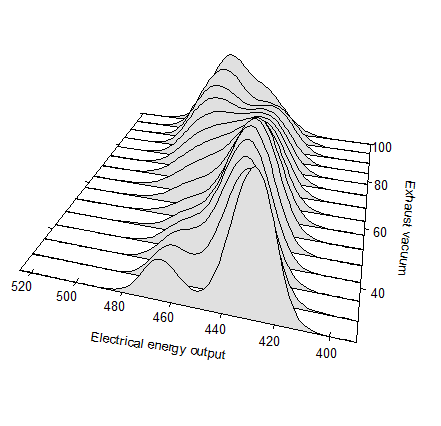}
\label{fig:power plant}
\caption{The conditional density estimate of the electronic energy output given Exhaust Vacuum.}
\end{figure}

To construct prediction intervals, we combine the proposed estimator with the  split conformal prediction of \cite{Lei2018}. Specifically:
\begin{enumerate}
\item[1.] Randomly split the index set $\{1,\dots,9568 \}$ into three parts $\mathcal{I}_{1},\mathcal{I}_{2}$, and $\mathcal{I}_{3}$. 
\item[2.] Use the data $\{ (Y_{i},X_{i}) : i \in \mathcal{I}_{1} \}$ to construct the estimator $\hat{m}(\cdot)$ for the modal function $m(\cdot)$.
\item[3.] Compute the $\alpha/2$- and $(1-\alpha/2)$-quantiles of $\{ Y_{i} - \hat{m}(X_{i}) : i \in \mathcal{I}_{2} \}$ and they are denoted by $\hat{\xi}_{\alpha/2}$ and $\hat{\xi}_{1-\alpha/2}$, respectively. In this experiment, $\alpha=0.05$ is used. 
\item[4.] Construct $C_{\text{split}}(x) = [ \hat{m}(x) + \hat{\xi}_{\alpha/2}, \hat{m}(x) + \hat{\xi}_{1-\alpha/2}]$.
\item[5.] Compute the empirical coverage probability:
\[
\frac{1}{|\mathcal{I}_{3}|} \sum_{i \in \mathcal{I}_{3}} I\{ Y_{i} \in C_{\text{split}}(X_{i}) \}.
\]
\end{enumerate}
In this experiment, we take $\mathcal{I}_{1},\mathcal{I}_{2}$, and $\mathcal{I}_{3}$ in such a way that $| \mathcal{I}_{1} \cup \mathcal{I}_{2} | : | \mathcal{I}_{3} | \approx 0.95 : 0.05$ and $| \mathcal{I}_{1} | : | \mathcal{I}_{2} | \approx 8:2$. 
We repeated this procedure 250 times and report the average of the empirical coverage probabilities together with the average and median lengths. In addition, we compare the proposed estimator with the linear modal regression estimator.
Table \ref{table:conformal} shows the results. 
For both methods, the empirical coverage probabilities are surprisingly close to the nominal coverage probability of $95 \%$, which is consistent with the theory developed in \cite{Lei2018}. On the other hand, the average and median lengths of the conformal prediction band with the proposed estimator are substantially smaller than those with the linear modal regression estimator, which is an encouraging sign for the proposed estimator. 

{\small
\begin{table}[h]
\centering
\begin{tabular}{|c|c|c|c|} \hline
Method &  Average length &Median length& Coverage probability \\ \hline \hline
Proposed method &19.01 &19.02 & 0.950 \\
Modal linear regression &23.71 &23.32 &0.950 \\
\hline
\end{tabular}
\medskip
\caption{Monte Carlo average and median lengths, and empirical coverage probabilities of the $95\%$ conformal prediction intervals.}
\label{table:conformal}
\end{table}
}

\section{Discussion}
\label{sec: conclusion}

In the present paper we have proposed a new estimator for the conditional mode based on quantile regression. 
The proposed estimate is computationally scalable since the quantile regression problem can be formulated as a linear programming problem. We have developed asymptotic distributional theory for the proposed estimator, which turns out to be nonstandard. Specifically, we have shown that the rate of convergence of the proposed estimator is $(nh^{2})^{1/3}$ where $h = h_n \to 0$ is a sequence of bandwidths, and that the limiting distribution is a scale transformation of Chernoff's distribution. For inference, we have discussed analytical and subsampling confidence intervals.
Finally we have verified the practical usefulness of the proposed method through numerical experiments. 

In the present paper, we use the naive quantile regression estimator $\hat{\beta}(\tau)$ that is not smooth in $\tau$ to estimate the conditional quantile function, while the true slope vector $\beta (\tau)$ is smooth in $\tau$ under our assumption. An interesting alternative approach is to impose smoothness to $\hat{\beta}(\tau)$ so that the estimated conditional quantile function is differentiable in $\tau$. We expect that the resulting conditional mode estimator would have a Gaussian limit (under regularity conditions), which is a reminiscent of the smoothed maximum score estimator of \cite{Horowitz1992}. Developing this alternative approach requires a whole new theory and is left as future research.

\section*{Acknowledgments}

The authors would like thank the Editor Domenico Marinucci, an AE, and an anonymous referee for their careful review and  constructive comments that helped improve on the quality of the paper. 

\clearpage

\appendix

\section{Proofs}

\subsection{Preliminaries}

In what follows, we will obey the following notation. 
For a given probability space $(S,\mathcal{S},Q)$ and a measurable function $f: S \to \R$, we use the notation $Qf = \int f dQ$ whenever the latter integral exists. 
For a class of measurable real-valued functions $\mathcal{F}$ on $S$, let $N(\mathcal{F},\| \cdot \|_{Q,2},\delta)$ denote the $\delta$-covering number for $\mathcal{F}$ with respect to the $L^{2}(Q)$-seminorm  $\| \cdot \|_{Q,2}$; see Section 2.1 in \cite{vdVW1996} for details. 
In addition, for a (vector-valued) function $g$ on a set $T$, we use the notation $\| g \|_{T} = \sup_{x \in T} \| g(x) \|$, where $\| \cdot \|$ denotes the Euclidean norm. 
We denote by $\stackrel{d}{=}$ the equality in distribution. 

The following maximal inequality will be repeatedly used in the proof of Theorem \ref{thm: limiting distributions}. 
 
\begin{lemma}[A useful maximal inequality]
\label{lem: maximal inequality}
Let $X_{1},\dots,X_{n}$ be i.i.d. random variables taking values in a measurable space $(S,\mathcal{S})$ with common distribution $P$, and let $\mathcal{F}$ be a pointwise measurable class of (measurable) real-valued functions on $S$ with measurable envelope $F$.\footnote{The class $\mathcal{F}$ is said to be pointwise measurable if there exists a countable subclass $\mathcal{G} \subset \mathcal{F}$ such that for every $f \in \mathcal{F}$ there exists a sequence $g_{m} \in \mathcal{G}$ with $g_{m} \to f$ pointwise; see Section 2.3 in \cite{vdVW1996}. } Suppose that there exist constants $A \geq e$ and $V \geq 1$ such that $\sup_{Q} N(\mathcal{F},\| \cdot \|_{Q,2}, \eta\| F \|_{Q,2}) \leq (A/\eta)^{V}$ for all $0 <  \eta \leq 1$, where $\sup_{Q}$ is taken over all finitely discrete distributions on $S$. Furthermore, suppose that $0 < PF^{2} < \infty$,
 and let $\sigma^{2}$ be any positive constant such that $\sup_{f \in \mathcal{F}} Pf^{2}\le \sigma^{2} \le PF^{2}$. Finally, let $B=\sqrt{\E[\max_{1 \le i \le n} F^{2}(X_{i})]}$. 
Then
\[
\E \left [ \left \|\sum_{i=1}^{n} \{ f(X_{i}) - Pf \} \right \|_{\mathcal{F}} \right ]   \le C  \left [ \sqrt{nV\sigma^{2}\log (A\| F \|_{P,2}/\sigma)} + VB \log (A \| F \|_{P,2}/\sigma) \right ],
\]
 where $\| \cdot \|_{\mathcal{F}} = \sup_{f \in \mathcal{F}} | \cdot |$ and $C$ is a universal constant. 
\end{lemma}

\begin{proof}
See Corollary 5.1 in \cite{CCK2014}.
\end{proof}

In particular, if we take $\sigma^{2} = PF^{2}$, then using the inequality $B \le \sqrt{n} \| F \|_{P,2}$, we also have 
\begin{equation}
\E \left [ \left \|\sum_{i=1}^{n} \{ f(X_{i}) - Pf \} \right \|_{\mathcal{F}} \right ]   \le 2C \sqrt{n}  \| F \|_{P,2} V \log A.
\label{eq: simple maximal inequality}
\end{equation}
The right hand side on (\ref{eq: simple maximal inequality}) can be improved to $\| F \|_{P,2} \sqrt{V \log A}$ up to a universal constant (cf. Theorem 2.14.1 in \cite{vdVW1996}), but this does not matter to the proof of Theorem \ref{thm: limiting distributions}.

\begin{lemma}
\label{lem: max convergence}
For i.i.d. random variables $\zeta_{1},\zeta_{2},\dots$, $\E[ \max_{1 \le i \le n}|\zeta_{i}|] = o(n)$ if and only if $\E[ |\zeta_1| ] < \infty$.
\end{lemma}

\begin{proof}
This is a well known result in probability theory, but we provide its proof for the sake of completeness. The ``only if'' direction is trivial, and so we prove the ``if'' direction. Suppose that $\E[|\zeta_1|] < \infty$. Then the strong law of large numbers yields that $\max_{1 \le i \le n} |\zeta_{i}|/n \le \sum_{i=1}^{n} |\zeta_i|/n \to \E[|\zeta_1|]$ almost surely, which also implies that $\max_{1 \le i \le n}|\zeta_{i}|/n \to 0$ almost surely (in general for a sequence of real numbers $\{ a_{i} \}_{i=1}^{\infty}$, if $n^{-1} \sum_{i=1}^{n}a_{i}$ converges as $n \to \infty$, then $\max_{1 \le i \le n} |a_{i}| = o(n)$). The the desired result follows from the generalized dominated convergence theorem (cf. Problem 4.3.12 in \cite{Dudley2002}). 
\end{proof}


\subsection{Proof of Theorem \ref{thm: limiting distributions}}

The proof of Theorem \ref{thm: limiting distributions} depends on the following Bahadur representation of the quantile regression estimator $\hat{\beta}(\tau)$. 

\begin{lemma}[Bahadur representation of $\hat{\beta}(\tau)$]
\label{lem: Bahadur}
Under Assumption \ref{assumption}, we have 
\[
\hat{\beta}(\tau) - \beta(\tau) = J(\tau)^{-1} \left [ \frac{1}{n} \sum_{i=1}^{n}\{ \tau - I(Y_{i} \le X_{i}^{T}\beta(\tau)) \} X_{i}\right ] + R_{n}(\tau),
\]
where $J(\tau) = \E [ f(X^{T}\beta(\tau) \mid X)XX^{T} ]$ and $\| R_n \|_{[\varepsilon/2,1-\varepsilon/2]} = o_{\Prob}(n^{-3/4} \log n)$. In addition, 
\begin{equation}
\left \| \frac{1}{n}\sum_{i=1}^{n}\{ \tau - I(Y_{i} \le X_{i}^{T}\beta(\cdot)) \} X_{i} \right \|_{[\varepsilon/2,1-\varepsilon/2]} = O_{\Prob} (n^{-1/2}).
\label{eq: score process}
\end{equation}
\end{lemma}

The conclusion of the lemma is partly known in the literature, but we include the proof of the lemma since we could not find a right reference that exactly establishes the conclusion of the lemma under our assumption. We defer the proof of this lemma after the proof of Theorem \ref{thm: limiting distributions}. 

\begin{proof}[Proof of Theorem \ref{thm: limiting distributions}]
We divide the proof into several steps. 

\underline{Step 1}. 
We first expand the objective function $\hat{s}_{x}(\tau)$ using the Bahadur representation of $\hat{\beta}(\tau)$. 
Let $F(y \mid X)$ denote the conditional distribution function of $Y$ given $X$, and let $U_{i} = F(Y_i \mid X_i)$ for $i=1,\dots,n$. The variable $U_{i}$ follows the uniform distribution on $(0,1)$ independent of $X_{i}$ for each  $i=1,\dots,n$. Since 
\[
Y_{i} \le X_{i}^{T}\beta(\tau) \Leftrightarrow U_i \le \tau
\]
under our assumption (recall that $X_{i}^{T} \beta(\tau)$ is the conditional $\tau$-quantile of $Y_{i}$ given $X_{i}$), 
we also have 
\begin{equation}
\hat{\beta}(\tau) - \beta(\tau) = J(\tau)^{-1} \left [ \frac{1}{n}\sum_{i=1}^{n}\{ \tau - I(U_i \le \tau) \} X_{i}\right ] + R_{n}(\tau). \label{eq: Bahadur}
\end{equation}
Using the Bahadur representation (\ref{eq: Bahadur}) along with some calculations, we have that 
\[
\begin{split}
\hat{s}_{x}(\tau) &= s_{x,n}(\tau) + x^{T} J(\tau)^{-1} \left [ \frac{1}{n} \sum_{i=1}^{n} \{1-I(U_{i} \in (\tau-h,\tau+h])/(2h) \} X_{i} \right ]  \\
&\quad + \underbrace{O_{\Prob}(n^{-1/2}) + o_{\Prob}(n^{-3/4}h^{-1} \log n)}_{=o_{\Prob}((nh^{2})^{-2/3})},
\end{split}
\]
where $s_{n,x} = \{ Q_{x}(\tau+h)-Q_{x}(x-h) \}/(2h)$ and the $o_{\Prob}$ and $O_{\Prob}$ terms are uniform in $\tau \in [\varepsilon,1-\varepsilon]$. 

Now, let $K(u)=I(u \in (-1,1])/2$ and $K_h(u) =h^{-1}K(u/h)$. 
Define 
\[
g_{n,\tau} (U,X) = s_{x,n}(\tau) + x^{T}J(\tau)^{-1}X \{ 1-K_{h}(U-\tau) \}.
\]
Denoting by $\Prob_{n}$ the empirical probability measure for $\{ (U_{i},X_{i}) \}_{i=1}^{n}$, we have 
\[
\hat{s}_{x}(\tau) = \Prob_{n} g_{n,\tau} + o_{\Prob}((nh^{2})^{-2/3}),
\]
where the $o_{\Prob}$ term is uniform in $\tau \in [\varepsilon,1-\varepsilon]$, and so $\hat{\tau}_{x}$ satisfies that 
\begin{equation}
\Prob_{n} g_{n,\hat{\tau}_{x}} \le \inf_{\tau \in [\varepsilon,1-\varepsilon]} \Prob_{n} g_{n,\tau} + o_{\Prob}((nh^{2})^{-2/3}). \label{eq: optimality condition}
\end{equation}
In what follows, we denote by $P$ the joint distribution of $(U,X)$.

\underline{Step 2}. Next, we show consistency of $\hat{\tau}_{x}$. 
To this end, consider the function class $\mathcal{G}_{n} = \{ g_{n,\tau} : \tau \in [\varepsilon,1-\varepsilon] \}$. It is seen that there exists a constant $C_1$ (independent of $n$) such that $\sup_{\tau \in [\varepsilon,1-\varepsilon]}|g_{n,\tau}(U,X)| \le C_{1}(1+\| X \|/h) =: G_{n}(U,X)$.
 Then there exist constants $A_{1}$ and $V_{1}$ independent of $n$ such that 
\[
\sup_{Q} N(\mathcal{G}_{n},\| \cdot \|_{Q,2}, \eta \| G_{n} \|_{Q,2}) \le (A_{1}/\eta)^{V_{1}}, \ 0 < \forall \eta \le 1,
\]
where the $\sup_{Q}$ is taken over all finitely discrete distributions on $(0,1) \times \mathcal{X}$. This follows from a small modification to the proof of Lemma 3.1 in \cite{Ghosal2000} and so we omit the detailed proof. In addition, it is seen that $\sup_{\tau \in [\varepsilon,1-\varepsilon]} Pg_{n,\tau}^{2} = O(h^{-1}), \ PG_{n}^{2} = O(h^{-2})$, and $\E[ \max_{1 \le i \le n} G_{n}^{2}(U_{i},X_{i})] = o(n^{1/2}h^{-2})$ by Lemma \ref{lem: max convergence}.

Now, applying the maximal inequality of Lemma \ref{lem: maximal inequality}, we have 
\begin{equation}
\label{eq: rate sparsity}
\E \left [ \| \Prob_{n}g_{n,\tau} - Pg_{n,\tau} \|_{[\varepsilon,1-\varepsilon]} \right ] = \underbrace{O((nh)^{-1/2} \sqrt{\log n}) + o( n^{-3/4}h^{-1} \log n)}_{= o(1)},
\end{equation}
which implies that $\| \Prob_{n}g_{n,\tau} - Pg_{n,\tau} \|_{[\varepsilon,1-\varepsilon]} = o_{\Prob} (1)$ by Markov's inequality. 
Further, $Pg_{n,\tau} = s_{x,n}(\tau) = s_{x}(\tau) + o(1)$ uniformly in $\tau \in [\varepsilon,1-\varepsilon]$ and $s_{x}(\tau)$ is uniquely minimized at $\tau = \tau_{x}$ by assumption. 
Hence, by Theorem 5.7 in \cite{vanderVaart2000}, we have $\hat{\tau}_{x} \stackrel{\Prob}{\to} \tau_{x}$. 

\underline{Step 3}. The aim of this step is to show that $\hat{\tau}_{x} = \tau_{x} + O_{\Prob}((nh^{2})^{-1/3})$. We divide this step into three sub-steps. 

\underline{Step 3-(a)}. We begin with observing that, for any $\delta = \delta_{n} \to 0$, $Pg_{n,\tau}$ can be expanded as 
\[
P g_{n,\tau} = s_{x,n}(\tau) = s_{x,n}(\tau_{x}) + s_{n,x}'(\tau_{x}) (\tau-\tau_{x}) + (s_{x,n}''(\tau_{x}) /2+ o(1)) (\tau - \tau_{x})^{2}
\]
uniformly in $| \tau - \tau_{x} | < \delta$, 
and $s_{n,x}'(\tau_{x}) = \{ Q_{x}'(\tau_{x}+h) - Q_{x}'(\tau_{x}-h)\}/(2h)= O(h^{2})$,
where we have used the fact that $Q_{x}''(\tau_{x}) = s_{x}'(\tau_{x}) = 0$ (recall that $\tau_{x}$ is a minimizer of $s_{x}(\tau)$). Indeed, recalling that  $Q_{x}(\tau)$ is four times continuously differentiable in $\tau$, we have 
\[
\begin{split}
Q_{x}'(\tau_{x} + h) &= Q_{x}'(\tau_{x}) + \underbrace{Q_{x}''(\tau_{x})}_{=0} h + \frac{Q_{x}'''(\tau_{x})}{2} h^{2} + O(h^{3}),  \ \text{and likewise} \\
Q_{x}'(\tau_{x} - h) &= Q_{x}'(\tau_{x})  + \frac{Q_{x}'''(\tau_{x})}{2} h^{2} + O(h^{3}),
\end{split}
\]
which implies that $\{ Q_{x}'(\tau_{x}+h) - Q_{x}'(\tau_{x}-h)\}/(2h)= O(h^{2})$. 
Since $h^{2} = o((nh^{2})^{-1/3})$, 
using the inequality $|ab| \le (a^{2}+b^{2})/2$, we have
\[
| s_{n,x}'(\tau_{x}) (\tau-\tau_{x}) | \le  o(1) (\tau - \tau_{x})^{2} + o((nh^{2})^{-2/3}). 
\]
Further, $s_{x,n}''(\tau_{x}) = s_{x}''(\tau_{x})+ o(1)$, and so we have 
\begin{equation}
P (g_{n,\tau} - g_{n,\tau_{x}})  = (v_{x} + o(1)) (\tau - \tau_{x})^{2}  + o((nh^{2})^{-2/3})
\label{eq: identification}
\end{equation}
uniformly in $| \tau - \tau_{x} |  < \delta$, where $v_{x} = s_{x}''(\tau_{x})/2 > 0$. 

\underline{Step 3-(b)}. 
Next, for given $\delta > 0$, consider the function class $\mathcal{G}_{n,\delta} = \{ g_{n,\tau} - g_{n,\tau_{x}} : \tau \in [\varepsilon,1-\varepsilon], | \tau - \tau_{x} | < \delta \}$. 
It is seen that there exists a constant $C_{2}$ independent of $n$ and $\delta$ such that, whenever $|\tau - \tau_{x}| < \delta$,
\begin{equation}
\begin{split}
&| g_{n,\tau}(U,X) - g_{n,\tau_{x}} (U,X) | \\
&\le C_{2}\left [ \{ (1+\| X \|/h) \delta + (\| X \|/h) \{I(|U - \tau_{x}+h| \le \delta) + I(|U-\tau_{x}-h| \le \delta) \}  \right ] \\
&=: G_{n,\delta}(U,X).
\end{split}
\label{eq: envelope}
\end{equation}
Then there exist constants $A_{2}$ and $V_{2}$ independent of $n$ and $\delta$ such that 
\begin{equation}
\sup_{Q} N(\mathcal{G}_{n,\delta}, \| \cdot \|_{Q,2}, \eta \| G_{n,\delta} \|_{Q,2}) \le (A_{2}/\eta)^{V_{2}}, \ 0 < \forall \eta \le 1.
\label{eq: VC type}
\end{equation}
Again, this follows from a small modification to the proof of Lemma 3.1 in \cite{Ghosal2000}.

\underline{Step 3-(c)}. Finally, 
by consistency of $\hat{\tau}_{x}$, there exists $\delta = \delta_{n} \to 0$ such that $\Prob (| \hat{\tau}_{x}  - \tau_{x}| < \delta_{n}) \to 1$. 
In view of the expansion (\ref{eq: identification}), for sufficiently large $n$, we have
\[
P(g_{n,\tau} - g_{n,\tau_{x}}) \ge v_{x}(\tau-\tau_{x})^{2}/2-o((nh^{2})^{-2/3})
\]
uniformly in $| \tau - \tau_{x} | < \delta$. 
Further, by the covering number estimate of Step 3-(b) together with the maximal inequality (\ref{eq: simple maximal inequality}), we have
\[
\E\left [ \left \| \Prob_{n} g - Pg \right \|_{\mathcal{G}_{n,\delta}} \right ] = O(n^{-1/2}h^{-1}\delta^{1/2}),
\]
where we have used the fact that $PG_{n,\delta}^{2} =O(h^{-2}\delta)$. 
Now, a small modification to the proof of Theorem 3.2.5 in \cite{vdVW1996} shows that $| \hat{\tau}_{x} - \tau_{x}| = O_{\Prob}(r_{n}^{-1})$, where $r_{n}$ satisfies $r_{n}^{2}h^{-1}r_{n}^{-1/2} = n^{1/2}$, i.e., $r_{n} = (nh^{2})^{1/3}$. This completes Step 3.

\underline{Step 4}. 
Let $a_{n} = (nh^{2})^{1/3}$, and define
\[
\check{g}_{n,t}  = 
\begin{cases}
n^{1/6}h^{4/3} (g_{n,\tau_{x}+t/a_{n}} - g_{n,\tau_{x}}) & \text{if} \ \tau_{x}+t/a_{n} \in [\varepsilon,1-\varepsilon] \\
0 & \text{otherwise}
\end{cases}
. 
\]
Consider the empirical process
\[
\mathbb{G}_{n} \check{g}_{n,t} := \sqrt{n} (\Prob_{n} \check{g}_{n,t} - P\check{g}_{n,t}), \ t \in \R.
\]
Recall that $\sigma_{x}^{2} = \E[(x^{T}J(\tau_{x})^{-1} X)^{2}]/2$. 
The aim of this step is to show weak convergence of the empirical process $\{ \mathbb{G}_{n} \check{g}_{n,t} : t \in \R \}$ to $\{ \sigma_{x} B(t) : t \in \R \}$ in $\ell^{\infty}_{\mathrm{loc}}(\R)$, where 
$\ell^{\infty}_{\mathrm{loc}}(\R)$ is the space of locally bounded functions on $\R$ equipped with the metric $d(f,g) = \sum_{N=1}^{\infty} 2^{-N} (1 \wedge \| f-g \|_{[-N,N]})$; cf. Section 1.6 in \cite{vdVW1996}. This reduces to verifying (i) the finite dimensional convergence, i.e., for any $t_{1},\dots,t_{\ell} \in \R$,
\[
\left (\mathbb{G}_{n}\check{g}_{n,t_{1}},\dots,\mathbb{G}_{n}\check{g}_{n,t_{\ell}} \right )
\stackrel{d}{\to}\left (\sigma_{x}B(t_{1}),\dots,\sigma_{x}B(t_{\ell}) \right );
\]
 and (ii)  the asymptotic equicontinuity of the empirical process on $[-N,N]$ for each $N=1,2,\dots$, i.e., for any $\eta > 0$, 
\begin{equation}
\lim_{\delta \to 0} \limsup_{n \to \infty} \Prob \left ( \sup_{\substack{|t_{1}-t_{2}| < \delta \\ t_{1},t_{2} \in [-N,N]}} | \mathbb{G}_{n}(\check{g}_{n,t_{1}} - \check{g}_{n,t_{2}}) | > \eta \right ) = 0. 
\label{eq: asymptotic equicontinuity}
\end{equation}

To verify the finite dimensional convergence, we first compute the limit of the covariance of $\check{g}_{n,t_1}$ and $\check{g}_{n,t_2}$ for $t_1 \le t_2$. To this end,
let
\[
\varphi_{n,t} (U,X) = n^{1/6}h^{4/3}x^{T}J(\tau_{x})^{-1} X\{ K_{h}(U-\tau_{x}) - K_{h}(U-\tau_{x}-t/a_{n}) \}.
\]
Direct (but tedious) calculations show that $\Cov_{P} (\check{g}_{n,t_{1}},\check{g}_{n,t_{2}}) =P(\varphi_{n,t_{1}}\varphi_{n,t_{2}}) + o(1)$, where $\Cov_{P}$ denotes the covariance under $P$.
Since $X$ and $U$ are independent, we focus on computing 
\begin{equation}
\label{eq: covariance}
\begin{split}
&\E[ \{ K_{h}(U-\tau_{x}) - K_{h}(U-\tau_{x}-t_1/a_{n}) \} \{ K_{h}(U-\tau_{x}) - K_{h}(U-\tau_{x}-t_2/a_{n}) \}] \\
&=\frac{1}{4h^{2}} \Big (2h - \big| [(\tau_{x}+t_1/a_{n}) \pm h] \cap [\tau_{x} \pm h] \big| - \big| [(\tau_{x}+t_2/a_{n}) \pm h] \cap [\tau_{x} \pm h] \big| \\
&\qquad + \big| [(\tau_{x}+t_1/a_{n}) \pm h] \cap [(\tau_{x}+t_2/a_{n}) \pm h] \big| \Big ),
\end{split}
\end{equation}
where $[a \pm b] = [a-b,a+b]$ and $| \cdot |$ denotes the Lebesgue measure.
First, since $ha_n = (nh^{5})^{1/3} \to \infty$, for sufficiently large $n$, we have
\[
\big| [(\tau_{x}+t/a_{n}) \pm h] \cap [\tau_{x} \pm h] \big| = 2h -\frac{|t|}{a_{n}}.
\]
Next, if $t_{1} \le t_{2}$, then for sufficiently large $n$, we have
\[
\big| [(\tau_{x}+t_1/a_{n}) \pm h] \cap [(\tau_{x}+t_2/a_{n}) \pm h] \big| = 2h - \frac{t_{2}-t_{1}}{a_{n}}.
\]
Combining these estimates leads to
\[
\begin{split}
&2h - \big| [(\tau_{x}+t_1/a_{n}) \pm h] \cap [\tau_{x} \pm h] \big| - \big| [(\tau_{x}+t_2/a_{n}) \pm h] \cap [\tau_{x} \pm h] \big| \\
&\qquad + \big| [(\tau_{x}+t_1/a_{n}) \pm h] \cap [(\tau_{x}+t_2/a_{n}) \pm h] \big| \\
&=
\begin{cases}
\frac{2t_{1}}{a_{n}} & \text{if} \ 0 \le t_{1} \le t_{2} \\
\frac{-2t_{2}}{a_{n}} & \text{if} \ t_{1} \le t_{2} \le 0 \\
0 & \text{if} \ t_{1} \le 0 \le t_{2}
\end{cases}
.
\end{split}
\]
Since $a_{n}h^{2} = n^{1/3} h^{8/3}$, we conclude that 
\[
\lim_{n \to \infty}\Cov_{P} (\check{g}_{n,t_{1}},\check{g}_{n,t_{2}}) = \sigma_{x}^{2} 
\E[B(t_1)B(t_2)]
.
\]
The rest is to verify the Lindeberg condition, and to this end it is enough to verify that for any $t \in \R$ and $\eta > 0$, 
\[
n^{1/3}h^{8/3}PG_{n,|t|/a_{n}}^{2}I(n^{1/6}h^{4/3}G_{n,|t|/a_{n}} > \eta \sqrt{n}) \to 0,
\]
where $G_{n,\delta}$ is given in (\ref{eq: envelope}). After a few more calculations, we see that the problem boils down to showing that
\begin{equation}
a_{n} \E\left [ \| X \|^{2} I(|U-\tau_{x} \pm h| \le |t|/a_n) I(\| X \| > \eta n^{1/3}h^{-1/3}) \right ] \to 0.
\label{eq: Lindeberg condition}
\end{equation}
However, since $X$ and $U$ are independent, the left hand side on (\ref{eq: Lindeberg condition}) is
\[
\underbrace{a_{n} \Prob(|U-\tau_{x} \pm h| \le |t|/a_n)}_{=O(1)} \underbrace{\E \left [\| X \|^{2} I(\| X \| > \eta n^{1/3}h^{-1/3}) \right ]}_{=o(1)} \to 0.
\]
Therefore, we have proved the finite dimensional  convergence. 

To verify the asymptotic equicontinuity (\ref{eq: asymptotic equicontinuity}), consider the function class 
\[
\check{\mathcal{G}}_{n,\delta} = \left \{ \check{g}_{n,t_{1}} - \check{g}_{n,t_{2}} : | t_{1} - t_{2} | < \delta, t_{1},t_{2} \in [-N,N] \right \}.
\]
We will apply Lemma \ref{lem: maximal inequality} to the function class $\check{\mathcal{G}}_{n,\delta}$. First, 
an envelope function for $\check{\mathcal{G}}_{n,\delta}$ is given by $\check{G}_{n} = 2n^{1/6}h^{4/3}G_{n,N/a_{n}}$. Observe that, using independence between $U$ and $X$, we have
$P\check{G}_{n}^{2} = O(1)$ and 
\[
\E\left[ \max_{1 \le i \le n} \check{G}_{n}^{2}(U_{i},X_{i}) \right ] \le O(n^{1/3}h^{2/3}) \E \left [ \max_{1 \le i \le n} \| X_{i} \|^{2} \right ] = o(n^{5/6}h^{3/2}) = o(n),
\]
where we have used $\E \left [ \max_{1 \le i \le n} \| X_{i} \|^{2} \right ] = o(n^{1/2})$, which follows from Lemma \ref{lem: max convergence}. 

Next, from the covering number estimate (\ref{eq: VC type}),  there exist constants $A_{3}$ and $V_{3}$ independent of $n$ and $\delta$ such that 
\[
\sup_{Q} N(\check{\mathcal{G}}_{n,\delta}, \| \cdot \|_{Q,2}, \eta \| \check{G}_{n} \|_{Q,2}) \le (A_{3}/\eta)^{V_{3}}, \ 0 < \forall \eta \le 1.
\]

Finally, it is seen that there exists a constant $C_{3}$ independent of $n$ such that 
\[
P(g_{n,\tau_{1}} - g_{n,\tau_{2}})^{2} \le C_{3}|\tau_{1}-\tau_{2}|/h^{2}, \ \forall \tau_{1},\tau_{2} \in [\varepsilon,1-\varepsilon],
\]
which implies that 
\[
P(\check{g}_{n,t_{1}} - \check{g}_{n,t_{2}})^{2} \le C_{3}| t_{1} - t_{2} |, \ \forall t_{1},t_{2} \in [-N,N]
\]
for sufficiently large $n$. 

Therefore, applying Lemma \ref{lem: maximal inequality} to the function class $\check{\mathcal{G}}_{n,\delta}$, we conclude that there exists a constant $C_{4}$ independent of $n$ and $\delta$ such that 
\[
\E \left [ \sup_{\substack{|t_{1}-t_{2}| < \delta \\ t_{1},t_{2} \in [-N,N]}} | \mathbb{G}_{n}(\check{g}_{n,t_{1}} - \check{g}_{n,t_{2}}) | \right ] \le C_{4}\sqrt{\delta \log (1/\delta)} + o(1) \log(1/\delta)
\]
for sufficiently small $\delta$, where the $o(1)$ term is independent of $\delta$. 
This leads to the asymptotic equicontinuity (\ref{eq: asymptotic equicontinuity}) by Markov's inequality.

\underline{Step 5}. We derive the limit distribution of $\hat{\tau}_{x}$ by applying Theorem 2.7 in \cite{KimPollard1990}. 
The optimality condition (\ref{eq: optimality condition}) implies that the rescaled  estimator $\hat{t} = (nh^{2})^{1/3} (\hat{\tau}_{x} - \tau_{x})$ satisfies 
\[
\sqrt{n} \Prob_{n} (-\check{g}_{n,\hat{t}}) \ge \sup_{t \in \R} \sqrt{n} \Prob_{n} (-\check{g}_{n,t}) - o_{\Prob} (1).
\]
In view of the expansion (\ref{eq: identification}), we have 
\[
\sqrt{n}P\check{g}_{n,t} = v_{x}t^{2} + o(1)
\]
locally uniformly in $t \in \R$, i.e., uniformly in $t \in [-N,N]$ for each $N=1,2,\dots$. 
From the weak convergence result of Step 4, together with the fact that $B \stackrel{d}{=} -B$, the non-centered empirical process $\{ \sqrt{n} \Prob_{n} (-\check{g}_{n,t}) : t \in \R \}$ converges weakly to the process $\{ \sigma_{x}B(t) - v_{x}t^{2} : t \in \R \}$ in $\ell_{\mathrm{loc}}^{\infty}(\R)$, and the limit process concentrates on $C_{\max}(\R)$ (as defined in \cite{KimPollard1990}) by Lemmas 2.5 and 2.6 in \cite{KimPollard1990}. 
Further, $\hat{t} = O_{\Prob}(1)$ by Step 3. Therefore, by Theorem 2.7 in \cite{KimPollard1990}, we have
\[
\hat{t} = (nh^{2})^{1/3} (\hat{\tau}_{x} - \tau_{x}) \stackrel{d}{\to} \argmax_{t \in \R} \left \{ \sigma_{x} B(t) - v_{x} t^{2} \right \}.
\]
The right hand side is equal in distribution to $(\sigma_{x}/v_{x})^{2/3} Z$ by Problem 3.2.5 in \cite{vdVW1996}, where $Z = \argmax_{t \in \R} \{ B(t) - t^{2} \}$. This leads to the first result (\ref{eq: first result}) of the theorem.

Finally, observe that
\[
\hat{m}(x) - m(x) = \hat{Q}_{x}(\hat{\tau}_{x}) - Q_{x}(\tau_{x}) = \hat{Q}_{x}(\hat{\tau}_{x}) - Q_{x}(\hat{\tau}_{x}) + Q_{x}(\hat{\tau}_{x}) - Q_{x}(\tau_{x}).
\]
By Lemma \ref{lem: Bahadur}, 
\[
|\hat{Q}_{x}(\hat{\tau}_{x}) - Q_{x}(\hat{\tau}_{x})| \le \| \hat{Q}_{x} - Q_{x} \|_{[\varepsilon,1-\varepsilon]}
\le \| x \| \| \hat{\beta} - \beta \|_{[\varepsilon,1-\varepsilon]} = O_{\Prob}(n^{-1/2}).
\]
Applying the delta method, we have 
\[
(nh^{2})^{1/3} (\hat{m}(x) - m(x)) = (nh^{2})^{1/3}(Q_{x}(\hat{\tau}_{x}) - Q_{x}(\tau_{x})) + o_{\Prob}(1) \stackrel{d}{\to} s_{x}(\tau_{x})(\sigma_{x}/v_{x})^{2/3} Z.
\]
This completes the proof.
\end{proof}

\begin{proof}[Proof of Lemma \ref{lem: Bahadur}]
The results (\ref{eq: score process}) and $\| \hat{\beta} - \beta \|_{[\varepsilon/2,1-\varepsilon/2]}=O_{\Prob}(n^{-1/2})$ follow from Theorem 3 in \cite{ACF2006}.
By the first order condition for the quantile regression problem (\ref{eq: QR problem}), we have 
\begin{align}
&\left \| \sum_{i=1}^{n} \{ \tau - I(Y_{i} \le X_{i}^{T} \hat{\beta}(\tau)) \} X_{i} \right \| \le \Card (\{ i \in \{ 1,\dots,n \}: Y_{i} = X_{i}^{T} \hat{\beta}(\tau) \}) \max_{1 \le i \le n} \| X_{i} \|, \ \text{and} \label{eq: FOC} \\
&\sup_{\tau \in [\varepsilon/2,1-\varepsilon/2]} \Card (\{ i \in \{1,\dots,n \} : Y_{i} = X_{i}^{T} \hat{\beta}(\tau) \}) \le d \quad \text{almost surely}. \label{eq: fit}
\end{align}
The first result (\ref{eq: FOC}) follows from a modification to the proof of Lemma 2.1 in \cite{ElAttar1979}; see Lemma \ref{lem: FOC} ahead. The second result (\ref{eq: fit}) follows from the following observation. Pick any subset $I \subset \{ 1,\dots, n \}$ such that $\Card (I) \ge d+1$. Conditionally on $X_{1}^{n} = \{ X_{1},\dots,X_{n} \}$, consider the set 
\[
S_{I} = \{ (X_{i}^{T} \beta)_{i \in I} : \beta \in \R^{d} \} \subset \R^{\Card (I)},
\]
which is a linear subspace of dimension at most $d$. 
If there exists $\tau \in [\varepsilon/2,1-\varepsilon/2]$ such that $Y_{i} = X_{i}^{T} \hat{\beta}(\tau)$ for all $i \in I$, then $(Y_{i})_{i \in I} \in S_{I}$, so that 
\begin{equation}
\begin{split}
&\Prob (\text{ there exists $\tau \in [\varepsilon/2,1-\varepsilon/2]$ such that $Y_{i} = X_{i}^{T} \hat{\beta}(\tau)$ for all $i \in I$} \mid X_{1}^{n}) \\
&\quad \le \Prob ( (Y_{i})_{i \in I} \in S_{I} \mid X_{1}^{n}). 
\end{split}
\label{eq: fit2}
\end{equation}
However, since  the distribution of $(Y_{i})_{i \in I}$ conditionally on $X_{1}^{n}$ is absolutely continuous, the conditional probability on the right hand side is $0$. 
By Fubini's theorem, the unconditional probability of the event inside the conditional probability on the left hand side of (\ref{eq: fit2}) is $0$. Now,
\[
\begin{split}
&\Prob\left(\sup_{\tau \in [\varepsilon/2,1-\varepsilon/2]} \Card (\{ i \in \{ 1,\dots, n \}: Y_{i} = X_{i}^{T} \hat{\beta}(\tau) \}) \ge d+1 \right)  \\
&\le \sum_{\substack{I \subset \{ 1,\dots, n \} \\ \Card (I) \ge d+1}} \Prob (\text{ there exists $\tau \in [\varepsilon/2,1-\varepsilon/2]$ such that $Y_{i} = X_{i}^{T} \hat{\beta}(\tau)$ for all $i \in I$}) = 0,
\end{split}
\]
which leads to the result (\ref{eq: fit}).

Since $\E[\| X \|^{4}] < \infty$, we have $\max_{1 \le i \le n} \| X_{i}\| = o_{\Prob}(n^{1/4})$ (cf. Lemma \ref{lem: max convergence}), and so
\[
\left \| \frac{1}{n}\sum_{i=1}^{n} \{ \tau - I(Y_{i} \le X_{i}^{T} \hat{\beta}(\cdot)) \} X_{i} \right \|_{[\varepsilon/2,1-\varepsilon/2]} = o_{\Prob}(n^{-3/4}).
\]
We will expand $n^{-1} \sum_{i=1}^{n} \{ \tau - I(Y_{i} \le X_{i}^{T} \hat{\beta}(\tau)) \} X_{i}$. Observe that 
\[
\begin{split}
&\frac{1}{n} \sum_{i=1}^{n} \{ \tau - I(Y_{i} \le X_{i}^{T} \hat{\beta}(\tau)) \} X_{i}
= \frac{1}{n} \sum_{i=1}^{n} \{ \tau - I(Y_{i} \le X_{i}^{T} \beta (\tau)) \} X_{i} + \E[  \{ \tau - I(Y \le X^{T} \beta) \} X]|_{\beta = \hat{\beta}(\tau)} \\
&\qquad + \frac{1}{n} \sum_{i=1}^{n}\{ I(Y_{i} \le X_{i}^{T}\beta(\tau)) - I(Y_{i} \le X_{i}^{T}\hat{\beta}(\tau)) \} X_{i} -  \E[  \{ \tau - I(Y \le X^{T} \beta) \} X]|_{\beta = \hat{\beta}(\tau)}
\end{split}
\]
The Taylor expansion yields that 
\[
\E[  \{ \tau - I(Y \le X^{T} \beta) \} X]|_{\beta = \hat{\beta}(\tau)} = -J(\tau)(\hat{\beta}(\tau) - \beta(\tau)) + O_{\Prob}(n^{-1})
\]
uniformly in $\tau \in [\varepsilon/2,1-\varepsilon/2]$. 
It remains to show that 
\begin{equation}
\begin{split}
&\left \|  \frac{1}{n} \sum_{i=1}^{n}\{ I(Y_{i} \le X_{i}^{T}\beta(\tau)) - I(Y_{i} \le X_{i}^{T}\hat{\beta}(\tau)) \} X_{i} -  \E[  \{ \tau - I(Y \le X^{T} \beta) \} X]|_{\beta = \hat{\beta}(\tau)} \right \|_{[\varepsilon/2,1-\varepsilon/2]} \\
&\quad = o_{\Prob}(n^{-3/4} \log n). 
\label{eq: residual process}
\end{split}
\end{equation}
Since $\| \hat{\beta} - \beta \|_{[\varepsilon/2,1-\varepsilon/2]} = O(n^{-1/2})$, for any $M_{n} \to \infty$ sufficiently slowly, $\Prob (\| \hat{\beta} - \beta \|_{[\varepsilon/2
,1-\varepsilon/2]} \le M_{n}n^{-1/2}) \to 1$.
Consider the function class
\[
\mathcal{F}_{n} = \left\{ (y,x) \mapsto \{ I(y \le x^{T}\beta) - I(y \le x^{T}(\beta + \delta)) \}\alpha^{T}x : \beta \in \R^{d}, \| \delta \| \le M_{n}n^{-1/2}, \alpha \in \mathbb{S}^{d-1} \right\}, 
\]
where $\mathbb{S}^{d-1} = \{ x \in \R^{d} : \| x \| = 1 \}$. Then the left side on (\ref{eq: residual process}) is bounded by 
\begin{equation}
\left \| \frac{1}{n} \sum_{i=1}^{n} f(Y_{i},X_{i}) - \E[f(Y,X)] \right \|_{\mathcal{F}_{n}}
\label{eq: residual process bound}
\end{equation}
with probability approaching one. 
Since the function class $\{ (y,x) \mapsto I(y \le x^{T}\beta) \alpha^{T}x : \beta \in \R^{d}, \alpha \in \mathbb{S}^{d-1} \}$ (that is independent of $n$) is a VC subgraph class with envelope $F(y,x) = \| x \|$, there exist constants $A$ and $V$ independent of $n$ such that 
\[
\sup_{Q} N(\mathcal{F}_{n}, \| \cdot \|_{Q,2}, \eta \| F \|_{Q,2}) \le (A/\eta)^{V}, \ 0 < \forall \eta \le 1.
\]
See Section 2.6 in \cite{vdVW1996}. Simple calculations show that 
\[
\begin{split}
&\sup_{f \in \mathcal{F}_{n}} \E[f^{2}(Y,X)] = O(M_{n}n^{-1/2}) \quad \text{and} \\
&\E\left[\max_{1 \le i \le n} F^{2}(Y_{i},X_{i}) \right] = \E\left[ \max_{1 \le i \le n}\| X_{i} \|^{2}\right] = o(n^{1/2})
\end{split}
\]
by Lemma \ref{lem: max convergence}. Therefore, applying Lemma \ref{lem: maximal inequality} to the function class $\mathcal{F}_{n}$ shows that the expectation of the term (\ref{eq: residual process bound}) is bounded by
\[
O(n^{-3/4} \sqrt{M_{n} \log n}) + o(n^{-3/4} \log n).
\]
Choosing $M_{n} \to \infty$ sufficiently slowly, we obtain the desired result. 
\end{proof}

\begin{lemma}
\label{lem: FOC}
Let $(y_1,x_1),\dots,(y_n,x_n) \in \R \times \R^{d}$ be pairs of outcome variables and regressors. Consider to solve the quantile regression problem: 
\begin{equation}
\min_{\beta \in \R^{d}} \sum_{i=1}^{n} \rho_{\tau}(y_{i}-x_{i}^{T}\beta),
\label{eq: QR}
\end{equation}
where $\tau \in (0,1)$ is fixed. Let $\beta^{*}$ be an optimal solution to (\ref{eq: QR}) and let $I^{*} = \{ i \in \{1,\dots, n\}: y_{i} = x_{i}^{T}\beta^{*} \}$. Then there exist $a_{i} \in [-1,0]$ for $i \in I^{*}$ such that 
\[
\sum_{i=1}^{n} \{ \tau - I(y_{i} \le x_{i}^{T}\beta^{*}) \}x_{i} = \sum_{i \in I^{*}} a_{i}x_{i}.
\]
Hence we have $\| \sum_{i=1}^{n} \{ \tau - I(y_{i} \le x_{i}^{T}\beta^{*}) \}x_{i} \| \le \Card (I^{*}) \max_{1 \le i \le n} \| x_{i} \|$.
\end{lemma}
\begin{proof}
Let $y=(y_{1},\dots,y_{n})^{T}$ and $\mathbb{X} = [x_{1},\dots,x_{n}]^{T}$.  The optimization problem (\ref{eq: QR}) reduces to the following linear programming problem:
\begin{equation}
\label{eq: QRLP}
\begin{split}
&\min_{u, v \in \R^{n}, \beta \in \R^{d}} \tau  1_{n}^{T}  u + (1-\tau)  1_{n}^{T}  v \\
&\quad \text{s.t.} \  u- v= y- \mathbb{X}\beta, \  u \ge  0_{n}, \  v \ge  0_{n}, 
\end{split}
\end{equation}
where $1_{n} = (1,\dots,1)^{T} \in \R^{n}$ and $0_{n}=(0,\dots,0)^{T} \in \R^{n}$. The inequalities $u \ge  0_{n}$ and $v \ge  0_{n}$ are  interpreted coordinatewise. 
Let $u_{i}^{*} = \max \{ y_{i} -  x_{i}^{T}  \beta^{*}, 0 \}$ and $v_{i}^{*} = \max \{ -y_{i} +  x_{i}^{T}  \beta^{*}, 0 \}$.
Then $u^{*} -  v^{*}=  y -  \mathbb{X}\beta^{*}$ and $(u^{*},v^{*},\beta^{*})$ is an optimal solution to the problem (\ref{eq: QRLP}). 
Defining
\begin{align*}
&f(u,v,\beta) = \tau  1_{n}^{T}u + (1-\tau)1_{n}^{T}v, \\
&g(u,v,\beta) = (g_{1}(u,v,\beta),\dots,g_{2n}(u,v,\beta))^{T}=(-u^{T},-v^{T})^{T}, \\
&h(u,v,\beta)=(h_{1}(u,v,\beta),\dots,h_{n}(u,v,\beta))^{T}= u-v-y+\mathbb{X}\beta,
\end{align*}
the problem (\ref{eq: QRLP}) can be written as
\begin{align*}
&\min_{u,v\in \R^{n},\beta \in \R^{d}} f(u,v,\beta) \\
&\quad \text{s.t.} \  g(u,v,\beta) \le {0}_{2n}, \  h(u,v,\beta) = 0_{n}.
\end{align*}
Let $e_{i} \in \R^{n}$ denote the vector of which only the $i$-th element is $1$ and the other elements are all zero. 
Then the gradient vectors of $f(u,v,\beta), \ g_{i}(u,v,\beta), \ g_{n+i}(u,v,\beta)$, and $h_{i}(u,v,\beta)$ are given by
\begin{align*}
&\nabla f(u,v,\beta) = 
\begin{pmatrix}
\tau  1_{n} \\
(1-\tau) 1_{n} \\
0_{d}
\end{pmatrix}, \
\nabla g_{i}( u, v, \beta) = 
\begin{pmatrix}
- e_{i} \\
 0_{n} \\
 0_{d}
\end{pmatrix}, \\
&\nabla g_{n+i}( u, v, \beta) =
\begin{pmatrix}
 0 \\
- e_{i} \\
 0_{d}
\end{pmatrix}, \
\nabla h_{i}( u, v, \beta) = 
\begin{pmatrix}
 e_{i} \\
- e_{i} \\
 x_{i} 
\end{pmatrix}, \ i=1,\dots,n.
\end{align*}
Since all the constraints are linear, by the Karush-Kuhn-Tucker theorem (cf. \cite{Bertsekas1999}, Proposition 3.3.7), there exist $\mu_{1},\dots,\mu_{2n} \ge 0$ and $\lambda_{1},\dots,\lambda_{n} \in \R$ such that 
\begin{align}
&\begin{pmatrix}
\tau  1_{n} \\
(1-\tau) 1_{n} \\
0_{d}
\end{pmatrix}
+ 
\sum_{i=1}^{n} \mu_{i} 
\begin{pmatrix}
- e_{i} \\
 0_{n} \\
 0_{d}
\end{pmatrix}
+
\sum_{i=1}^{n} \mu_{n+i}
 \begin{pmatrix}
 0_{n} \\
- e_{i} \\
 0_{d}
\end{pmatrix}
+
\sum_{i=1}^{n} \lambda_{i}
\begin{pmatrix}
 e_{i} \\
- e_{i} \\
 x_{i} 
\end{pmatrix}
=
 0_{2n+d}, \label{eq: KKT} \\
 &\mu_{i} u^{*}_{i} = 0, \quad \text{and} \quad \mu_{n+i} v^{*}_{i}=0, \ i=1,\dots,n. \label{eq: slack}
\end{align}

Recall that $I^{*} = \{ i \in \{ 1,\dots, n \} : y_{i}= x_{i}^{T}  \beta^{*} \}$. Let $I_{+}^{*} = \{ i \in \{1,\dots,n \} : y_{i} >  x_{i}^{T}  \beta^{*} \}$ and $I_{-}^{*} = \{ i \in \{ 1,\dots,n \} : y_{i} <  x_{i}^{T}  \beta^{*} \}$.
Observe that from the complementary slack condition (\ref{eq: slack}), 
\begin{align*}
&i \in I_{+}^{*} \Rightarrow u_{i}^{*} > 0 \Rightarrow \mu_{i}=0 \Rightarrow \lambda_{i}=-\tau \quad \text{and} \\
&i \in I_{-}^{*} \Rightarrow v_{i}^{*} > 0 \Rightarrow \mu_{n+i}=0 \Rightarrow \lambda_{i}=1-\tau.
\end{align*}
The last $d$ equations in (\ref{eq: KKT}) imply that $\sum_{i=1}^{n} \lambda_{i} x_{i} = 0$, which can be rearranged as $\tau  \sum_{i \in I_{+}^{*}}  x_{i} + (\tau -1) \sum_{i \in I_{-}^{*}}  x_{i} = \sum_{i \in I^{*}} \lambda_{i} x_{i}$.
The left hand side is 
\[
\sum_{i \in I^{*}_{+} \cup I_{-}^{*}} \{ \tau - I(y_{i} \le  x_{i}^{T}  \beta^{*}) \}  x_{i} = \sum_{i=1}^{n} \{ \tau - I(y_{i} \le  x_{i}^{T}  \beta^{*}) \}  x_{i}  + (1-\tau) \sum_{i \in I^{*}}  x_{i},
\]
so that 
\[
\sum_{i=1}^{n} \{ \tau - I(y_{i} \le  x_{i}^{T}  \beta^{*}) \}  x_{i} = \sum_{i \in I^{*}} \underbrace{( \lambda_{i} -1+\tau)}_{=a_{i}}  x_{i}.
\]
For $i \in I^{*}$, we have by the first $2n$ equations of (\ref{eq: KKT}), 
\begin{align*}
&\tau - \mu_{i} + \lambda_{i} = 0 \Rightarrow \lambda_{i} \ge - \tau \quad \text{and} \\
&1-\tau-\mu_{n+i} - \lambda_{i} = 0 \Rightarrow \lambda_{i} \le 1 - \tau,
\end{align*}
so that $\lambda_{i} \in [-\tau,1-\tau]$ for $i \in I^{*}$, i.e, $a_{i} \in [-1,0]$ for $i \in I^{*}$. This completes the proof. 
\end{proof}

\subsection{Proof of Corollary \ref{cor: multivariate}}

The second result follows from the delta method (see the proof of Theorem \ref{thm: limiting distributions}), so we focus on proving the first result. 
We will follow the notation used in the proof of Theorem \ref{thm: limiting distributions}, but to make the dependence on $x$ explicit, 
let us write $g_{n,x,\tau} (U,X) = s_{x,n} (\tau) + x^{T} J(\tau)^{-1} X\{ 1-K_{h}(U-\tau) \}$, 
\[
\check{g}_{n,x,t}  = 
\begin{cases}
n^{1/6}h^{4/3} (g_{n,x,\tau_{x}+t/a_{n}} - g_{n,x,\tau_{x}}) & \text{if} \ \tau_{x}+t/a_{n} \in [\varepsilon,1-\varepsilon] \\
0 & \text{otherwise}
\end{cases}
,
\]
and $\varphi_{n,x,t} (U,X) = n^{1/6}h^{4/3} x^{T}J(\tau_{x})^{-1}X\{ K_{h}(U-\tau_{x}) - K_{h}(U-\tau_{x} -t/a_{n}) \}$. 
Recall that $a_{n} = (nh^{2})^{1/3}$. 

We begin with observing that for $\hat{t}_{j} = (nh^{2})^{1/3} (\hat{\tau}_{x^{j}} - \tau_{x^{j}}), j=1,\dots,L$, 
\[
\begin{split}
\sqrt{n}\Prob_{n} \left (-\sum_{j=1}^{L} \check{g}_{n,x^{j},\hat{t}_{j}}  \right ) &\ge \sup_{(t_{1},\dots,t_{L})^{T} \in \R^{L}} \sqrt{n} \Prob_{n} \left (-\sum_{j=1}^{L} \check{g}_{n,x^{j},t_{j}}  \right )  - o_{\Prob}(1), \quad \text{and} \\
\sqrt{n}P \left ( \sum_{j=1}^{L} \check{g}_{n,x^{j},t_{j}} \right ) &= \sum_{j=1}^{L} v_{x^{j}} t_{j}^{2} + o(1) \quad \text{locally uniformly in $(t_{1},\dots,t_{L})^{T} \in \R^{L}$}. 
\end{split}
\]
In addition, from Theorem \ref{thm: limiting distributions}, we know that $\hat{t}_{j} = O_{\Prob}(1)$ for each $j=1,\dots,L$. 
Hence, in view of Theorem 2.7 in \cite{KimPollard1990}, we only have to verify the following. Let $\ell^{\infty}_{\mathrm{loc}}(\R^{L})$  denote the space of all locally bounded functions on $\R^{L}$ equipped with the metric $d(f,g) = \sum_{N=1}^{\infty} 2^{-N} (1 \wedge \| f -g \|_{[-N,N]^{L}})$. Recall that $\mathbb{G}_{n} g = \sqrt{n}(\Prob_n g - Pg)$. 
\begin{enumerate}
\item[(i)] There exists a continuous version of $\mathbb{B}_{k}$ for each $k=1,\dots,M$, and the stochastic process $\{ \mathbb{G}_{n}(\sum_{j=1}^{L} \check{g}_{n,x^{j},t_{j}}) : (t_{1},\dots,t_{L})^{T} \in \R^{L} \}$ converges weakly to the process $\{ \sum_{k=1}^{M} \mathbb{B}_{k}((t_{j})_{j \in S_{k}}) : (t_{1},\dots,t_{L})^{T} \in \R^{L} \}$ in $\ell_{\mathrm{loc}}^{\infty}(\R^{L})$, where $\mathbb{B}_{1},\dots,\mathbb{B}_{M}$ are independent.
\item[(ii)] For each $k=1,\dots,M$, the process 
\[
(t_{j})_{j \in S_{k}} \mapsto \mathbb{B}_{k}((t_{j})_{j \in S_{k}}) -\sum_{j \in S_{k}} v_{x^{j}}t_{j}^{2}
\]
admits a unique maximizer almost surely. 
\end{enumerate}
The latter (ii) follows from Lemmas 2.5 and 2.6 in \cite{KimPollard1990}, so we focus on verifying the weak convergence (i). By Section 1.6 in \cite{vdVW1996}, this boils down to verifying the finite dimensional convergence together with the asymptotic equicontinuity on each $[-N,N]^{L}$, i.e., for any $\eta > 0$,
\begin{equation}
\lim_{\delta \to 0} \limsup_{n \to \infty} \Prob \left (  \sup_{\substack{|t_{j} - t_{j}'| < \delta \\ t_{j},t_{j}' \in [-N,N], j=1,\dots,L} } \left| \mathbb{G}_{n}\left ( \sum_{j=1}^{L} (\check{g}_{n,x^{j},t_{j}} - \check{g}_{n,x^{j},t_{j}'}) \right ) \right | > \eta \right ) = 0. 
\label{eq: AEC}
\end{equation}
As we will see, the finite dimensional convergence and the asymptotic equicontinuity automatically imply the existence of a continuous version of $\mathbb{B}_{k}$ for each $k=1,\dots,M$. 

The asymptotic equicontinuity (\ref{eq: AEC}) follows from the fact that $|\mathbb{G}_{n}(\sum_{j=1}^{L} (\check{g}_{n,x^{j},t_{j}} - \check{g}_{n,x^{j},t_{j}'}))| \le \sum_{j=1}^{L} |\mathbb{G}_{n}(\check{g}_{n,x^{j},t_{j}} - \check{g}_{n,x^{j},t_{j}'})|$ and what we have proved in Step 4 in the proof of Theorem \ref{thm: limiting distributions}. It remains to prove the finite dimensional convergence. 
Direct calculations show that 
\[
\Cov_{P}\left ( \sum_{i=1}^{L} \check{g}_{n,x^{i},t_{i}}, \sum_{j=1}^{L} \check{g}_{n,x^{j},t_{j}'} \right ) = \sum_{i,j=1}^{L} P (\varphi_{n,x^{i},t_{i}} \varphi_{n,x^{j},t_{j}'}) + o(1)
\]
for any $(t_{1},\dots,t_{L})^{T}, (t_{1}',\dots,t_{L}')^{T} \in \R^{L}$. 
Consider first the case where $\tau_{x^{i}} = \tau_{x^{j}} = \tau_{(k)}$ for some $k=1,\dots,M$. Then, from the calculation done in Step 4 in the proof of Theorem \ref{thm: limiting distributions}, we see that 
\[
\lim_{n \to \infty} P (\varphi_{n,x^{i},t_{i}} \varphi_{n,x^{j},t_{j}'}) = \frac{1}{2}(x^{i})^{T} J(\tau_{(k)})^{-1} \E[XX^{T}] J(\tau_{(k)})^{-1} x^{j} \E[B(t_{i})B(t_{j}')]. 
\]
Next, consider the case where $\tau_{x^{i}} \neq \tau_{x^{j}}$. Then, the intervals $[\tau_{x^{i}} \pm h]$ and $[(\tau_{x^{i}} + t_{i}/a_{n}) \pm h]$ have empty intersections with $[\tau_{x^{j}} \pm h]$ and $[(\tau_{x^{j}} + t_{j}/a_{n}) \pm h]$ for sufficiently large $n$, so that 
\[
\lim_{n \to \infty} P (\varphi_{n,x^{i},t_{i}} \varphi_{n,x^{j},t_{j}'}) = 0.
\]
Conclude that 
\begin{equation}
\begin{split}
&\lim_{n \to \infty}\Cov_{P}\left ( \sum_{i=1}^{L} \check{g}_{n,x^{i},t_{i}}, \sum_{j=1}^{L} \check{g}_{n,x^{j},t_{j}'} \right ) \\
\quad &=\frac{1}{2} \sum_{k=1}^{M} \sum_{i,j \in S_{k}}(x^{i})^{T} J(\tau_{(k)})^{-1} \E[XX^{T}] J(\tau_{(k)})^{-1} x^{j} \E[B(t_{i})B(t_{j}')].
\label{eq: fidi}
\end{split}
\end{equation}
The Lindeberg condition can be verified in a similar way to Step 4 in the proof of Theorem \ref{thm: limiting distributions}, so we have proved the finite dimensional convergence. 

Now, for each $k=1,\dots,M$,  since $\mathbb{G}_{n}(\sum_{j=1}^{L} \check{g}_{n,x^{j},t_{j}}) \big |_{t_{j}=0, j \notin S_{k}} = \mathbb{G}_{n} (\sum_{j \in S_{k}} \check{g}_{n,x^{j},t_{j}})$, we see that the process $(t_{j})_{j \in S_{k}} \mapsto \mathbb{G}_{n} (\sum_{j \in S_{k}} \check{g}_{n,x^{j},t_{j}})$ is asymptotically equicontinuous  (with respect to the Euclidean metric) on $[-N,N]^{s_{k}}$ for each $N=1,2,\dots$ and the finite dimensional distributions converge weakly to those of $\mathbb{B}_{k}$. By the final paragraph in Section 1.6 of \cite{vdVW1996}, the limit process (in $\ell^{\infty}_{\mathrm{loc}} (\R^{s_{k}})$) is a version of $\mathbb{B}_{k}$ with continuous paths. 

We have already seen that the process $\{ \mathbb{G}_{n}(\sum_{j=1}^{L} \check{g}_{n,x^{j},t_{j}}) : (t_{1},\dots,t_{L})^{T} \in \R^{L} \}$ is weakly convergent in $\ell_{\mathrm{loc}}^{\infty}(\R^{L})$. The rest is to verify that the limit process is $\{ \sum_{k=1}^{M} \mathbb{B}_{k}((t_{j})_{j \in S_{k}}) : (t_{1},\dots,t_{L})^{T} \in \R^{L} \}$ where $\mathbb{B}_{1},\dots,\mathbb{B}_{M}$ are independent, which however follows from the fact that the right hand side on (\ref{eq: fidi}) is identical to $\Cov(\sum_{k=1}^{M} \mathbb{B}_{k}((t_{i})_{i \in S_{k}}),\sum_{k=1}^{M} \mathbb{B}_{k}((t_{j}')_{j \in S_{k}}))$. 
This completes the proof. \qed

\subsection{Proof of Proposition \ref{prop: inference}}

The consistency of $\hat{s}_{x}(\hat{\tau}_{x})$ follows from the uniform consistency of $\hat{s}_{x}(\tau)$ on $[\varepsilon,1-\varepsilon]$, i.e., $\| \hat{s}_{x} - s_{x} \|_{[\varepsilon,1-\varepsilon]} \stackrel{\Prob}{\to} 0$, which is established in Steps 1 and 2 in the proof of Theorem \ref{thm: limiting distributions}, together with the consistency of $\hat{\tau}_{x}$. 
Next, $\hat{\Sigma}$ is trivially consistent, and $\hat{J}(\tau)$ is uniformly consistent on $[\varepsilon,1-\varepsilon]$ by Section A.4 in \cite{ACF2006}. 
Together with the consistency of $\hat{\tau}_{x}$ and continuity of the map $\tau \mapsto J(\tau)$, we obtain the consistency of $\hat{\sigma}_{x}^{2}$. 
Finally, observe that $\Delta_{h}^{3}\hat{Q}_{x}(\tau) = \Delta_{h}^{2}\hat{s}_{x}(\tau)$, and $\hat{s}_{x}(\tau) = \Delta_{h}Q_{x}(\tau) + O_{\Prob}((nh)^{-1/2}\sqrt{\log n})$ uniformly in $\tau \in [2\varepsilon/3,1-2\varepsilon/3]$ by (\ref{eq: rate sparsity}), so that $\Delta_{h}^{3}\hat{Q}_{x}(\tau) = \Delta_{h}^{3}Q_{x}(\tau) + O_{\Prob}((nh^{5})^{-1/2}\sqrt{\log n})$ uniformly in $\tau \in [\varepsilon,1-\varepsilon]$. 
The consistency of $\hat{v}_{x}$ then follows from the condition that $nh^{5}/\log n \to \infty$, 
continuity of the third derivative of $Q_{x}(\tau)$ at $\tau = \tau_{x}$, and the consistency of $\hat{\tau}_{x}$. This completes the proof. 
\qed

\section{Convergence of maximum of Chernoff random variables}
\label{sec: Gumbel}

In this appendix, we consider weak convergence of the maximum of independent Chernoff random variables. 
Let $Z_{1},\dots,Z_{n}$ be independent Chernoff random variables, and let $Z_{(n)} = \max_{1 \le i \le n}Z_{i}$ and $|Z|_{(n)} = \max_{1 \le i \le n}|Z_{i}|$.
Chernoff's distribution is known to be absolutely continuous, and denote its density by $f_{Z}$. In addition, let $F_{Z}$ denote the distribution function of Chernoff's distribution. 
An explicit form of $f_{Z}$ is unknown, but by  Corollary 3.4 of \cite{Groeneboom1989},  the tail behavior of $f_{Z}$ is given by
\begin{equation}
f_{Z}(z) \sim 2 \lambda |z| e^{-\frac{2}{3} |z|^{3} - \kappa |z|}, \quad |z| \to \infty, 
\label{eq: density}
\end{equation}
where $\lambda$ and $\kappa$ are positive constants whose explicit values can be found in \cite{Groeneboom1989}. 
The precise meaning of (\ref{eq: density}) is  that the ratio of the left and right hand sides approaches one as $|z| \to \infty$. This implies that 
\begin{equation}
1-F_{Z} (z) \sim \frac{\lambda}{z} e^{-\frac{2}{3}z^{3} -\kappa z}, \quad z \to \infty.
\label{eq: Chernoff tail}
\end{equation}
Cf. Lemma 2.1 in \cite{Ho1998}. The following lemma shows that both $Z_{(n)}$ and $|Z|_{(n)}$ converge in distribution to the Gumbel distribution as $n \to \infty$ after normalization. 
This lemma gives a supporting result for Remark \ref{rem: uniform rate}, but is of independent interest. 
Recall that the (standard) Gumbel distribution is a distribution on $\R$ with distribution function  $\Lambda (z) = e^{-e^{-z}}$. 

\begin{lemma}
\label{lem: Chernoff maximum}
Let 
\[
a_{n} = 3 \left ( \frac{2}{3} \right )^{1/3} (\log n)^{2/3}, \ b_{n} =  \left(\frac{3}{2}\log n \right) ^{1/3} - \frac{1}{a_{n}} \left [\kappa \left ( \frac{3}{2} \log n \right)^{1/3} + \frac{1}{3} \log \log n + \frac{1}{3} \log \frac{3}{2} - \log \lambda  \right ],
\]
and define $b_{n}'$ by replacing $\lambda$ by $2\lambda$ in the definition of $b_{n}$. Then we have for any $z \in \R$, 
\[
\lim_{n \to \infty}\Prob (a_{n} (Z_{(n)} - b_{n}) \le z) = e^{-e^{-z}} \quad \text{and} \quad \lim_{n \to \infty}\Prob (a_{n} (|Z|_{(n)} - b_{n}') \le z) = e^{-e^{-z}}.
\]
\end{lemma}

We note that \cite{Ho1998} already point out that Chernoff's distribution is in the domain of attraction of the Gumbel distribution (see \cite{Ho1998} p.219), but they do not derive explicit norming constants. 

The proof follows from the tail behavior of the Chernoff survival function (\ref{eq: Chernoff tail}) combined with the following lemma.

\begin{lemma}
\label{lem: extreme}
Let $X_{1},X_{2},\dots \sim F$ i.i.d. for some distribution function $F$, and let $X_{(n)} = \max_{1 \le i \le n}X_{i}$. 
For a given constant $\tau \ge 0$ and a given sequence $u_{n}$, we have 
\[
n (1-F(u_{n})) \to \tau \Leftrightarrow \Prob (X_{(n)} \le u_{n}) \to e^{-\tau}.
\]
\end{lemma}

\begin{proof}[Proof of Lemma \ref{lem: extreme}]
See \cite{Leadbetter1983} Theorem 1.5.1. 
\end{proof}

\begin{proof}[Proof of Lemma \ref{lem: Chernoff maximum}]
We first consider $Z_{(n)}$. Fix any $z \in \R$ and define $u_{n}$ by $n(1-F_{Z}(u_{n})) = e^{-z}$. Then by the preceding lemma we have $\lim_{n \to \infty} \Prob (Z_{(n)} \le u_{n}) = e^{-e^{-z}}$. We will find an explicit value of $u_{n}$. 
By (\ref{eq: Chernoff tail}), $u_{n}$ satisfies 
\[
\frac{n \lambda}{u_{n}} e^{z - \frac{2}{3}u_{n}^{3} - \kappa zu_{n}} \to 1.
\]
Taking logarithms of both sides, we have 
\begin{equation}
\log n + \log \lambda + z - \frac{2}{3} u_{n}^{3} - \kappa u_{n} - \log u_{n} = o(1).
\label{eq: asymptotic}
\end{equation}
Among the last three terms on the left hand side of (\ref{eq: asymptotic}), $\frac{2}{3} u_{n}^{3}$ is the dominant term, so that 
\begin{equation}
\frac{\frac{2}{3} u_{n}^{3}}{\log n} \to 1.
\label{eq: asymptotic2}
\end{equation}
Taking logarithms of both sides, we also have 
\[
\log u_{n} = \frac{1}{3} \left [ \log \log n + \log \frac{3}{2} \right ] + o(1).
\]
Plugging this into (\ref{eq: asymptotic}), we have 
\[
\frac{2}{3} u_{n}^{3}  = \log n + z -\kappa u_{n} - \frac{1}{3} \log \log n -  \frac{1}{3} \log \frac{3}{2} + \log \lambda + o(1).
\]
In addition, (\ref{eq: asymptotic2}) also implies that 
\[
u_{n} = \left ( \frac{3}{2} \log n \right )^{1/3} + \delta_{n} \quad \text{with} \ \delta_{n} = o((\log n)^{1/3}).
\]
Plugging this into the preceding equation, using the identity $(a+b)^{3}= a^{3} + 3a^{2}b + 3ab^{2} + b^{3}$, and comparing the orders, we see that $\delta_{n} = o(1)$. Conclude that 
\[
u_{n}^{3} = \left ( \frac{3}{2} \log n \right ) \left [ 1+ \frac{z-\kappa \left ( \frac{3}{2} \log n \right)^{1/3} -\frac{1}{3} \log \log n - \frac{1}{3} \log \frac{3}{2} + \log \lambda}{\log n} + o((\log n)^{-1}) \right ].
\]
Using $(1+x)^{1/3} = 1+x/3+O(x^{2})$ as $x \to 0$, we have 
\[
\begin{split}
u_{n} &= \left ( \frac{3}{2} \log n \right )^{1/3}  \left [ 1+ \frac{z-\kappa \left ( \frac{3}{2} \log n \right)^{1/3} -\frac{1}{3} \log \log n - \frac{1}{3} \log \frac{3}{2} + \log \lambda}{3 \log n} + o((\log n)^{-1}) \right ] \\
&=a_{n}^{-1}z + b_{n} + o(a_{n}^{-1}). 
\end{split}
\]
Therefore, we have $\Prob (Z_{(n)} \le u_{n}) = \Prob (a_{n}(Z_{(n)} - b_{n}) \le z + o(1))$, which leads to the desired result for $Z_{(n)}$. 

The proof for $|Z|_{(n)}$ is completely analogous, since by the symmetry of Chernoff's distribution, the distribution function $G_{Z}$ of $|Z|$ is $G_{Z}(z) =2F_{Z}(z) - 1$, so that $1-G_{Z}(z) = 2(1-F_{Z}(z))$. 
\end{proof}

\bibliographystyle{plain}
\bibliography{modal}

\begin{thebibliography}{10}

\bibitem{AbrevayaHuang2005}
J.~Abrevaya and J.~Huang.
\newblock On the bootstrap of the maximum score estimator.
\newblock {\em Econometrica}, 73(4):1175--1204, 2005.

\bibitem{ACF2006}
J.~Angrist, V.~Chernozhukov, and I.~Fern{\'a}ndez-Val.
\newblock Quantile regression under misspecification, with an application to
  the {US} wage structure.
\newblock {\em Econometrica}, 74(2):539--563, 2006.

\bibitem{BCK2016}
A.~Belloni, V.~Chernozhukov, and K.~Kato.
\newblock Valid post-selection inference in high-dimensional approximately
  sparse quantile regression models.
\newblock {\em Journal of the American Statistical Association}, 2018.
\newblock To appear.

\bibitem{Bertsekas1999}
D.~Bertsekas.
\newblock {\em Nonlinear Programming (2nd Edition)}.
\newblock Athena Scientific, 1999.

\bibitem{Chacon2018}
J.E. Chac\'{o}n.
\newblock The modal age of statistics.
\newblock arXiv:1807.02789, 2018.

\bibitem{Chaudhuri1997}
P.~Chaudhuri, K.~Doksum, and A.~Samarov.
\newblock On average derivative quantile regression.
\newblock {\em Annals of Statistics}, 25:715--744, 1997.

\bibitem{Chen2017}
Y.-C. Chen.
\newblock Modal regression using kernel density estimation: A review.
\newblock arXiv:1710.07004, 2017.

\bibitem{Chen2016}
Y.-C. Chen, C.R. Genovese, R.J. Tibshirani, and L.~Wasserman.
\newblock Nonparametric modal regression.
\newblock {\em Annals of Statistics}, 44(2):489--514, 2016.

\bibitem{Chernoff1964}
H.~Chernoff.
\newblock Estimation of the mode.
\newblock {\em Annals of the Institute of Statistical Mathematics},
  16(1):31--41, 1964.

\bibitem{CCK2014}
V.~Chernozhukov, D.~Chetverikov, and K.~Kato.
\newblock Gaussian approximation of suprema of empirical processes.
\newblock {\em Annals of Statistics}, 42(4):1564--1597, 2014.

\bibitem{Dudley2002}
R.M. Dudley.
\newblock {\em Real Analysis and Probability}.
\newblock Cambridge University Press, 2002.

\bibitem{Durot2012}
C.~Durot, V.N. Kulikov, and H.P. Lopuha{\"a}.
\newblock The limit distribution of the {$L^{\infty}$}-error of {G}renander
  type estimators.
\newblock {\em Annals of Statistics}, 40:1578--1608, 2012.

\bibitem{Durot2018}
C.~Durot and H.P. Lopuha{\"a}.
\newblock Limit theory in monotone function estimation.
\newblock {\em Statistical Science}, 33:547--567, 2018.

\bibitem{EinbecTutz2006}
J.~Einbeck and G.~Tutz.
\newblock Modelling beyond regression functions: an application of multimodal
  regression to speed--flow data.
\newblock {\em Journal of the Royal Statistical Society Series C},
  55(4):461--475, 2006.

\bibitem{ElAttar1979}
R.A. El-Attar, M.~Vidyasagar, and S.P.K. Dutta.
\newblock An algorithm for $l_1$-norm minimization with application to
  nonlinear $l_1$-approximation.
\newblock {\em SIAM Journal on Numerical Analysis}, 16(1):70^^e2^^80^^9386,
  1979.

\bibitem{Feng2017}
Y.~Feng, J.~Fan, and J.A.K. Suykens.
\newblock A statistical learning approach to modal regression.
\newblock arXiv:1702.05960, 2017.

\bibitem{Ghosal2000}
S.~Ghosal, A.~Sen, and A.W. van~der Vaart.
\newblock Testing monotonicity of regression.
\newblock {\em Annals of Statistics}, 28(4):1054--1082, 2000.

\bibitem{Grenander1956}
U.~Grenander.
\newblock On the theory of mortality measurement: {P}art {II}.
\newblock {\em Scandinavian Actuarial Journal}, 39:125--153, 1956.

\bibitem{Groeneboom1989}
P.~Groeneboom.
\newblock Brownian motion with a parabolic drift and airy functions.
\newblock {\em Probability Theory and Related Fields}, 81:79--109, 1989.

\bibitem{GroeneboomWellner2001}
P.~Groeneboom and J.A. Wellner.
\newblock Computing {C}hernoff's distribution.
\newblock {\em Journal of Computational and Graphical Statistics},
  10(2):388--400, 2001.

\bibitem{Ho2017}
C.~Ho, P.~Damien, and S.~Walker.
\newblock Bayesian mode regression using mixtures of triangular densities.
\newblock {\em Journal of Econometrics}, 197(2):273--283, 2017.

\bibitem{Ho1998}
G.~Hooghiemstra and H.P. Lopuha{\"a}.
\newblock An extremal limit theorem for the argmax process of {B}rownian motion
  minus a parabolic drift.
\newblock {\em Extreme}, 1:215--240, 1998.

\bibitem{Horowitz1992}
J.L. Horowitz.
\newblock A smoothed maximum score estimator for the binary response model.
\newblock {\em Econometrica}, 60:505--531, 1992.

\bibitem{Kaya2012}
H.~Kaya, P.~Tufekci, and S.F. Gurgen.
\newblock Local and global learning methods for predicting power of a combined
  gas \& steam turbine.
\newblock {\em Proceedings of the International Conference on Emerging Trends
  in Computer and Electronics Engineering ICETCEE 2012}, pages 13--18, 2012.

\bibitem{KempSantos2012}
G.C. Kemp and J.~Santos-Silva.
\newblock Regression towards the mode.
\newblock {\em Journal of Econometrics}, 170(1):92--101, 2012.

\bibitem{KhardaniYao2017}
S.~Khardani and A.F. Yao.
\newblock Non linear parametric mode regression.
\newblock {\em Communications in Statistics-Theory and Methods},
  46(6):3006--3024, 2017.

\bibitem{KimPollard1990}
J.~Kim and D.~Pollard.
\newblock Cube root asymptotics.
\newblock {\em Annals of Statistics}, 18(1):191--219, 1990.

\bibitem{Koenker2005}
R.~Koenker.
\newblock {\em Quantile Regression}.
\newblock Cambridge University Press, 2005.

\bibitem{KoenkerBassett1978}
R.~Koenker and G.~Bassett.
\newblock Regression quantiles.
\newblock {\em Econometrica}, 46(1):33--50, 1978.

\bibitem{KoenkerMachado1999}
R.~Koenker and J.A.F. Machado.
\newblock Goodness of fit and related inference processes for quantile
  regression.
\newblock {\em Journal of the American Statistical Association},
  94(448):1296--1310, 1999.

\bibitem{Kosorok2008}
M.~Kosorok.
\newblock Bootstrapping the {G}renander estimator.
\newblock In N.~Balakrishnan, E.~Pena, and M.~Silvapulle, editors, {\em Beyond
  Parametrics in Interdisciplinary Research: Festschrift in Honour of Professor
  Pranab K. Sen}, pages 282--292. IMS, 2008.

\bibitem{Krief2017}
J.M. Krief.
\newblock Semi-linear mode regression.
\newblock {\em Econometrics Journal}, 20(2):149--167, 2017.

\bibitem{Leadbetter1983}
M.R. Leadbetter, G.~Lindgren, and H.~Rootz{\'e}n.
\newblock {\em Extremes and Related Properties of Random Sequences and
  Processes}.
\newblock Springer, 1983.

\bibitem{Lee1989}
M.-J. Lee.
\newblock Mode regression.
\newblock {\em Journal of Econometrics}, 42(3):337--349, 1989.

\bibitem{Lee1993}
M.-J. Lee.
\newblock Quadratic mode regression.
\newblock {\em Journal of Econometrics}, 57(1-3):1--19, 1993.

\bibitem{LegerMacGibbon2006}
C.~L\'{e}ger and B.~MacGibbon.
\newblock On the bootstrap in cube root asymptotics.
\newblock {\em Canadian Journal of Statistics}, 34(1):29--44, 2006.

\bibitem{Lei2018}
J.~Lei, M.~G'Sell, A.~Rinaldo, R.~Tibshirani, and L.~Wasserman.
\newblock Distribution-free predictive inference for regression.
\newblock {\em Journal of the American Statistical Association},
  113:1094--1111, 2018.

\bibitem{MaHe2016}
S.~Ma and X.~He.
\newblock Inference for single-index quantile regression models with profile
  optimization.
\newblock {\em Annals of Statistics}, 44:1234--1268, 2016.

\bibitem{Manski1975}
C.F. Manski.
\newblock Maximal score estimation of the stochastic utility model of choice.
\newblock {\em Journal of Econometrics}, 27(3):205--228, 1975.

\bibitem{Parzen1962}
E.~Parzen.
\newblock On estimation of a probability density function and mode.
\newblock {\em Annals of Mathematical Statistics}, 33(3):1065--1076, 1962.

\bibitem{PolitisRomano1994}
D.N. Politis and J.P. Romano.
\newblock Large sample confidence regions based on subsamples under minimal
  conditions.
\newblock {\em Annals of Statistics}, 22(4):2031--2050, 1994.

\bibitem{Politis1999}
D.N. Politis, J.P. Romano, and M.~Wolf.
\newblock {\em Subsampling}.
\newblock Springer, 1999.

\bibitem{Powell1986}
J.L. Powell.
\newblock Censored regression quantiles.
\newblock {\em Journal of Econometrics}, 32(1):143--155, 1986.

\bibitem{Romano1988}
J.P. Romano.
\newblock On weak convergence and optimality of kernel density estimates of the
  mode.
\newblock {\em Annals of Statistics}, 16(2):629--647, 1988.

\bibitem{SagerThisted1982}
T.W. Sager and R.A. Thisted.
\newblock Maximum likelihood estimation of isotonic modal regression.
\newblock {\em Annals of Statistics}, 10(3):690--707, 1982.

\bibitem{Sasaki2016}
H.~Sasaki, Y.~Ono, and M.~Sugiyama.
\newblock Modal regression via direct log-density derivative estimation.
\newblock In {\em International Conference on Neural Information Processing},
  pages 108--116, 2016.

\bibitem{Sen2010}
B.~Sen, M.~Banerjee, and M.~Woodroofe.
\newblock Inconsistency of bootstrap: {T}he {G}renander estimator.
\newblock {\em Annals of Statistics}, 38(4):1953--1977, 2010.

\bibitem{SeoOtsu2018}
M.H. Seo and T.~Otsu.
\newblock Local {M}-estimation with discontinuous criterion for dependent and
  limited observations.
\newblock {\em Annals of Statistics}, 46(1):344--369, 2018.

\bibitem{Tufekci2014}
P.~Tufekci.
\newblock Prediction of full load electrical power output of a base load
  operated combined cycle power plant using machine learning methods.
\newblock {\em International Journal of Electrical Power \& Energy Systems},
  60:126--140, 2014.

\bibitem{vanderVaart2000}
A.W. van~der Vaart.
\newblock {\em Asymptotic Statistics}.
\newblock Cambridge University Press, 2000.

\bibitem{vdVW1996}
A.W. van~der Vaart and J.A. Wellner.
\newblock {\em Weak Convergence and Empirical Processes: With Applications to
  Statistics}.
\newblock Springer, 1996.

\bibitem{WuYuYu2010}
T.Z. Wu, K.~Yu, and Y.~Yu.
\newblock Singe-index quantile regression.
\newblock {\em Journal of Multivariate Analysis}, 101:1607--1621, 2010.

\bibitem{YaoLi2014}
W.~Yao and L.~Li.
\newblock New regression model: modal linear regression.
\newblock {\em Scandinavian Journal of Statistics}, 41(3):656--671, 2014.

\bibitem{Yao2012}
W.~Yao, B.G. Lindsay, and R.~Li.
\newblock Local modal regression.
\newblock {\em Journal of Nonparametric Statistics}, 24(3):647--663, 2012.

\bibitem{ZhouHuang2016}
H.~Zhou and X.~Huang.
\newblock Nonparametric modal regression in the presence of measurement error.
\newblock {\em Electronic Journal of Statistics}, 10(2):3579--3620, 2016.

\end{thebibliography}

\end{document}